\documentclass[reqno]{amsart}
\usepackage{amsmath,amsxtra,amssymb,latexsym, amscd,amsthm}
\usepackage[mathscr]{euscript}
\usepackage{mathrsfs}
\usepackage{color}
\usepackage[active]{srcltx}
\usepackage{mathtools,stmaryrd}
\usepackage{enumerate}
\usepackage{hyperref}
\usepackage{framed}
\usepackage[T1]{fontenc}
\usepackage{subfigure}
\usepackage{multirow}
\usepackage{booktabs}
\usepackage{tikz}
\usepackage{longtable}
\usepackage{tikz-qtree}
\usepackage[edges]{forest}
\usetikzlibrary{trees} 
\usepackage{stmaryrd}
\SetSymbolFont{stmry}{bold}{U}{stmry}{m}{n}
\SetSymbolFont{stmry}{bold}{U}{stmry}{b}{n}

\newcommand{\RR}{{\mathbb R}}

\newcommand{\bS}{{\mathbb S}}

\newcommand{\N}{\mathbb{N}}

\newcommand{\K}{\mathbf{K}}
\newcommand{\Y}{\mathcal{Y}}

\newcommand{\bG}{\mathbf{G}}
\newcommand{\bH}{\mathbf{H}}
\newcommand{\bg}{\mathbf{g}}

\def\bx{{\boldsymbol{x}}}
\def\be{{\boldsymbol{e}}}
\def\by{{\boldsymbol{y}}}
\def\bz{{\boldsymbol{z}}}
\def\bw{{\boldsymbol{w}}}
\def\bv{{\boldsymbol{v}}}
\def\bu{{\boldsymbol{u}}}

\newcommand{\supp}{\mathrm{supp}}

\newcommand{\lram}[3]{\left\langle #1, #2\right\rangle_{#3}}
\newcommand{\trp}[1]{\mbox{\upshape tr}_p\left(#1\right)}

\def\bc{{\boldsymbol{c}}}
\def\bu{{\boldsymbol{u}}}

\def\bv{{\boldsymbol{v}}}
\def\bg{{\boldsymbol{\gamma}}}

\newcommand{\ud}{\mathrm{d}}
\newcommand{\td}[1]{\tilde{#1}}
\newcommand{\suppx}[2]{\mbox{supp}_{#1}(#2)}
\newcommand{\diag}[1]{\mbox{\upshape diag}(#1)}

\newcommand{\csp}{\mbox{\upshape\tiny cs}}
\newcommand{\tr}[1]{\mbox{\upshape tr}\left(#1\right)}
\def\ba{{\boldsymbol{\alpha}}}
\def\bb{{\boldsymbol{\beta}}}

\newcommand{\mI}{\mathcal{I}}

\newcommand{\mJ}{\mathcal{J}}
\newcommand{\wt}[1]{\widetilde{#1}}

\newcommand{\snd}[1]{\vert\N^n_{#1}\vert}

\newcommand{\QM}{\mathcal{Q}}

\newcommand{\X}{\mathcal{X}}

\newtheorem{assumption}{Assumption}[section]

\newtheorem{theorem}{Theorem}[section]

\newtheorem{corollary}{Corollary}[section]
\newtheorem{lemma}{Lemma}[section]
\newtheorem{example}{Example}[section]
\newtheorem{definition}{Definition}[section]
\newtheorem{remark}{Remark}[section]
\newtheorem{proposition}{Proposition}[section]

\title{Positivstellens\"atze for polynomial matrices with universal quantifiers}

\author{Feng Guo}
\address[Feng Guo]{School of Mathematical Sciences, Dalian  University of Technology, Dalian, 116024, China}
\email{fguo@dlut.edu.cn}

\author{Jie Wang}
\address[Jie Wang]{State Key Laboratory of Mathematical Sciences, Academy of Mathematics and Systems Science, Chinese Academy of Sciences, Beijing, 100190, China}
\email{wangjie212@amss.ac.cn}

\date{ \today}

\begin{document}

\begin{abstract}
This paper investigates Positivstellens{\"a}tze for polynomial matrices subject to universally quantified polynomial matrix inequality constraints. We first establish a matrix-valued Positivstellensatz under the Archimedean condition, incorporating universal quantifiers. For scalar-valued polynomial objectives, we further develop a sparse Positivstellensatz that leverages correlative sparsity patterns within these quantified constraints. Moving beyond the Archimedean framework, we then derive two generalized Positivstellens{\"a}tze under analogous settings. These results collectively unify and extend foundational theorems in three distinct contexts: classical polynomial Positivstellens{\"a}tze, their universally quantified counterparts, and matrix polynomial formulations. 
Applications of the established Positivstellens{\"a}tze to robust polynomial matrix inequality constrained 
optimization are also investigated.
\end{abstract}
\subjclass{90C23, 15A54, 13J30, 14P10, 11E25, 12D15}
\keywords{Positivstellensatz, sum of squares, polynomial matrix inequality, polynomial matrix optimization, universal quantifier, correlative sparsity}

\maketitle

\section{Introduction}

Positivstellens\"atze are fundamental results in real algebraic geometry, asserting under which conditions a polynomial is guaranteed to be positive on a given set 
\cite{realAG,Laurent_sumsof,PPSOSMarshall,ScheidererSurvey}. 
Most Positivstellens\"atze achieve this by expressing the polynomial using
sums of squares (SOS). 
These powerful results offer constructive methods for certifying whether a polynomial $f$ is positive over a basic semialgebraic set defined by polynomial inequalities
\[
\K\coloneqq\{\bx\in\RR^n \mid g_j(\bx)\ge 0, j=1,\ldots,s\},
\]
and have widespread applications in areas such as optimization, algebraic geometry, control theory, and more where polynomial positivity is a key concern \cite{SOCAGbook,HKL2020,Lasserre2015,Laurent_sumsof,Nie2023}.

% Positivstellens\"atze have been extensively studied in the literature.
Depending on whether the semialgebraic set $\K$ is assumed to be compact, 
Positivstellens\"atze are divided into two categories.
When $\K$ is a compact polyhedron with non-empty interior, Handelman's Positivstellensatz
\cite{Handelman1988} states
that a polynomial $f$ that is positive on $\K$ can be expressed as a positive linear combination of cross-products of $g_j$'s.
In the case of a more general compact basic semialgebraic set $\K$, Schm\"udgen \cite{Schmugen1991} demonstrated that if a polynomial $f$ is positive on $\K$, then it belongs to the preorder generated by $g_j$'s, that is, $f$ can be represented as an SOS-weighted combination of cross-products of $g_j$'s. 
Putinar’s Positivstellensatz \cite{Putinar1993} provides an alternative representation that avoids the need for cross-products of $g_j$'s under the Archimedean condition (slightly stronger than compactness). 
Building on Putinar’s Positivstellensatz, Lasserre \cite{LasserreGlobal2001,Lasserre09,Lasserre2015} 
introduced a moment-SOS hierarchy of semidefinite relaxations 
for polynomial optimization. This framework generates a non-decreasing sequence of lower bounds converging 
to the optimal value of a given polynomial optimization problem. 
%When $f$ is only nonnegative on $\K$, Lasserre and Netzer \cite{Lasserre05} provided an SOS approximation of $f$ via high-degree perturbations. 
% that gives an obvious certificate of nonnegativity of $f$ on $[-1,1]^n$. 
% They also extended this result 
% to give an approximation of polynomials that are nonnegative on a basic closed semialgebraic set.

There have been several attempts to provide certificates of positivity of a polynomial
over a non-compact semialgebraic set. Without requiring compactness of $\K$, 
Krivine-Stengle Positivstellensatz \cite{Krivine1964, Stengle1974} states that a polynomial $f$ is positive 
over $\K$ if and only if $\psi f=1+\phi$ for some $\psi, \phi$ from the preorder generated by $g_j$'s. P\'olya \cite{polya} demonstrated that if $f$ is homogeneous and positive on 
$\RR^n_+\setminus\{\mathbf{0}\}$, then 
multiplying $f$ by some power of $\sum_{i=1}^nx_i$, one obtains a polynomial with nonnegative coefficients. 
% \revision{Powers and Reznick \cite{PR2001} provided a ``positive'' version of P\'olya's theorem 
% asserting that all coefficients of $(\sum_{i=1}^nx_i)^rf$ are positive for $r$ sufficiently large.}
Dickinson and Povh \cite{DP2015} generalized P\'olya's Positivstellensatz for homogeneous polynomials being positive on the intersection 
$\RR^n_+\cap \K\setminus\{\mathbf{0}\}$, where $g_j$'s are assumed to be also homogeneous. 
Reznick \cite{Reznick1995} proved that, after multiplying by some power of $\sum_{i=1}^n x_i^2$, any 
positive definite (PD) form is a sum of even 
powers of linear forms. Putinar and Vasilescu \cite{PV1999} extended Reznick's result to the constrained case where $f$ and $g_j$'s are homogeneous polynomials of even degree.
% Mai et al. \cite{MLM2022} provided an alternative proof of Putinar and Vasilescu's
% result with an effective degree bound on the SOS weights in such certificates.
% Lasserre \cite{LasserreSOS2006} provided SOS approximations for polynomials that are nonnegative 
% on the whole $\RR^n$. 
When $\K$ is non-compact, SOS-structured positivity certificates of $f$ over $\K$ could be 
also established by investigating various geometric objects associated with the data, 
e.g., gradient varieties \cite{Nie2006}, principal gradient tentacles \cite{Markus06}, 
truncated tangency varieties \cite{Vui-constraints}, and the polar varieties \cite{GREUET2012503}. 

Recently, Hu, Klep, and Nie \cite{HKN2024} studied Positivstellens\"atze concerning semialgebraic sets defined by universal quantifiers (UQ). Specifically, for a given tuple $g=(g_1,\ldots,g_s)$ of real polynomials
in $\bx$ and $\by=(y_1,\ldots,y_m)$, and a closed set $\Y\subset\RR^m$, they consider 
representations of polynomials that are positive over the semialgebraic set 
\[
\mathcal{U}\coloneqq\{\bx\in\RR^n\mid g_1(\bx,\by)\ge 0, \ldots, g_s(\bx, \by)\ge 0, \ \forall \by\in\Y\}. 
\]
For a fixed measure $\nu$ with support exactly on $\Y$, under the Carleman condition on $\nu$ and the Archimedean condition,
they proved that if a polynomial $f$ is positive on $\mathcal{U}$, then $f$ belongs to the quadratic module associated to $(g, \nu)$. In other words, $f$ admits a representation
\[
\sigma_0(\bx)+\int_{\Y}\sigma_1(\bx,\by)g_1(\bx,\by)\ud\nu(\by)+\cdots+
\int_{\Y}\sigma_s(\bx,\by)g_s(\bx,\by)\ud\nu(\by),
\]
where each $\sigma_j$, $j=0,1,\ldots,s$, are SOS polynomials. They also investigated the 
corresponding moment problem on the semialgebraic set $\mathcal{U}$. As an important application, 
their results could be used to solve 
semi-infinite optimization problems which are highly challenging.

Most Positivstellens\"atze could be generalized to the matrix setting that both $f$ and $g_j$'s are polynomial matrices. 
Scherer and Hol \cite{SH2006} developed a matrix-version of Putinar's Positivstellensatz. 
Cimpri{\v c} \cite{CIMPRIC2011} extended the Krivine-Stengle Positivstellensatz to the case of polynomial matrices with polynomial constraints. 
A matrix version of Handelman's Positivstellensatz was proposed in \cite{LD2018}.
Building on Scherer and Hol's Positivstellensatz, Dinh et al. \cite{DHL2021} 
generalized the classical Schm\"udgen, 
Putinar-Vasilescu, and Dickinson–Povh Positivstellens\"atze to the polynomial matrix setting.
Given the projections of two semialgebraic sets defined by polynomial matrix inequalities (PMI), Klep and Nie \cite{KN2020} provided 
a matrix Positivstellensatz with lifting polynomials to determine whether one is contained in the other. These generalizations have a wide range of applications, particularly in areas such as optimal control, systems theory \cite{HL2006,HL2012,ichihara2009optimal,pozdyayev2014atomic,vanantwerp2000tutorial}.
% Additionally, several results have been established concerning Positivstellens\"atze for complex
% polynomials \cite{DP2009,Quillen1968}, and non-commutative polynomials
% \cite{BKP2016,CKP2012,CIMPRIC2013,Helton2002,HM2004,SAVCHUK2012}. 

Computing SOS-structured representations involved in Positivstellens\"atze could be typically cast as
semidefinite programs (SDP). 
% which can be efficiently addressed via interior-point methods
% \cite{Laurent_sumsof}.
However, the size of the SDPs grows rapidly with the dimension of the problem. 
Hence, from the perspective of computation, it becomes appealing to develop sparse versions of Positivstellens\"atze for sparse data, e.g., correlative sparsity \cite{MWG2024,LasserreSparsity,WKKM2006}, term sparsity \cite{MWG2024,magron2023sparse,wang3,wang2021tssos,wang2022cs}, and matrix (chordal) sparsity \cite{zheng2019chordal, ZF2023}. 

% In particular, as special cases of PMIs, linear or bilinear matrix inequality constrained problems appear 
% frequently in many synthesis problems for linear systems in optimal control \cite{vanantwerp2000tutorial}. 
Due to estimation errors or lack of information, 
the data from real-world problems often involve uncertainty. 
As a result, ensuring the robustness of PMIs over a prescribed set with uncertainty is crucial for some 
safety-critical applications with little tolerance for failure \cite{BBC2011}. 
Consequently, Positivstellens\"atze for polynomial matrices with UQs will be a powerful 
mathematical tool for addressing this issue, which serves as the primary motivation for this work.
Precisely, consider the semialgebraic set
\begin{equation}\label{defineX}
	\X\coloneqq\{\bx\in\RR^n \mid G(\bx, \by)\succeq 0, \ \forall  \by\in \Y\subset\RR^m\},
\end{equation}
where $G(\bx, \by)\in\bS[\bx,\by]^q$ (the set of $q\times q$ symmetric polynomial matrices 
in $\bx$ and $\by$), and $\Y\subset\RR^m$
is closed. The goal of this paper is to provide certificates for positive definiteness of a $p\times p$ symmetric polynomial matrix $F(\bx)\in\bS[\bx]^p$ over $\X$. 

\begin{table}\caption{Summary of Positivstellens\"atze for polynomials, polynomials with UQs, polynomial matrices, and  polynomial matrices with UQs}\label{tab1}
{\small
    \centering
    \begin{tabular}{c|c|c|c|c}
         \hline
         Positivstellensatz&Polynomials&Polynomials with UQs&PMIs&PMIs with UQs \\
         \hline
         Putinar&Putinar \cite{Putinar1993}&Hu, Klep, and Nie \cite{HKN2024}&Scherer and Hol \cite{SH2006}&Theorem \ref{th::SHpsatz}\\
         Putinar–Vasilescu & Putinar and Vasilescu \cite{PV1999}&Theorem \ref{th::homo}&Dinh et al. \cite{DHL2021}&Theorem \ref{th::homo}\\
         P{\'o}lya& P{\'o}lya \cite{polya}&Theorem \ref{th::PolyaMPU}&Theorem \ref{th::PolyaMPU}&Theorem \ref{th::PolyaMPU}\\
         \hline
    \end{tabular}
    }
\end{table}

\begin{table}\caption{Summary of sparse Positivstellens\"atze for polynomials, polynomial matrices, and polynomial matrices with UQs}\label{tab2}
    \centering
    \begin{tabular}{c|c}
         \hline
           Case& Literature\\
         \hline
         Polynomials with polynomial constraints &Lasserre \cite{LasserreSparsity}\\
         Polynomials with PMI constraints&Kojima and Muramatsu \cite{KM2009}\\
         Polynomial matrices with polynomial constraints&Counterexample \cite{MWG2024}\\
         Polynomials with PMI constraints and UQs&Theorem \ref{th::sparse}\\
         \hline
    \end{tabular}
\end{table}

\vskip 5pt
\noindent{\bf Contributions.} We generalize several classical Positivstellens\"atze from scalar polynomials to the matrix setting with UQs (see Tables \ref{tab1} and \ref{tab2} for summaries). 
% Some of these results remain novel and intriguing, even in the absence of the universal quantifiers. 
In the following, we highlight the main results of this paper.
Throughout the paper, let $\nu$ be a fixed Borel measure on $\RR^m$ with support $\supp(\nu)=\Y$ 
and satisfying $\int_{\Y}|h(\by)|\ud\nu(\by)<\infty$ for all $h(\by)\in\RR[\by]$.
For any $H(\bx, \by)\in\bS[\bx, \by]^p$, let us write 
$H(\bx, \by)=\sum_{\bb\in\suppx{\by}{H}}H_{\bb}(\bx)\by^{\bb}$ as a polynomial matrix in $\by$ with
coefficient matrices $H_{\bb}(\bx)\in\bS[\bx]^p$ and let 
\[
	\int_{\Y}
	H(\bx,\by)\ud\nu(\by)\coloneqq\sum_{\bb\in\suppx{\by}{H}}H_{\bb}(\bx)\int_{\Y}\by^{\bb}\ud\nu(\by)\in\bS[\bx]^p.
\]

The main result of this paper is the following matrix-valued Positivestellensatz incorporating universal quantifiers,
which holds under the Archimedean condition (Assumption \ref{assump2})
and the Carleman condition on $\nu$ (Assumption \ref{assump::carleman}).
See Section \ref{sec::pre} for the definition of the product $\langle\cdot,\cdot\rangle_p$. 

\vskip 6pt
\noindent{\bf Theorem A. (Theorem \ref{th::SHpsatz}) }
{\itshape Suppose that Assumptions \ref{assump2} and \ref{assump::carleman} hold. 
If $F(\bx)\in\bS[\bx]^{p}$ is PD on $\X$, then there exist
SOS matrices $S_0\in\bS[\bx]^p$, $S\in\bS[\bx,\by]^{pq}$ such that 
\[F(\bx)=S_0(\bx)+\int_{\Y}\lram{S(\bx,\by)}{G(\bx,\by)}{p}\ud\nu(\by).\]
}
\vskip 6pt

%We adopt a notably nontrivial operator-theoretic approach to prove Theorem A, 
%utilizing techniques such as the Gelfand-Naimark-Segal (GNS) construction for a carefully constructed 
%$\star$-algebra, and the spectral theory of self-adjoint operators in a Hilbert space.

Then we consider the particular case of $p=1$, namely, the objective is a scalar polynomial.
Suppose that $\X$ is defined by multiple PMIs with UQs ($G_j(\bx,\by)\succeq0,j=1,\ldots,s$). 
We assume the presence of correlative sparsity in the problem data, which implies that the variables $\bx$ decompose as a union of subsets $\bx=\cup_{\ell=1}^t\bx(\mI_\ell)$ such that $\{G_j\}_{j=1}^s=\cup_{\ell=1}^t\{G_j\}_{j\in\mJ_{\ell}}$
and $G_j\in\bS[\bx(\mI_\ell)]^{q_j}$, $j\in\mJ_{\ell}$, $\ell=1,\ldots,t$ (see Assumption~\ref{assump_CS}). 
Under these conditions, we could give the following sparse Positivstellensatz 
that leverages correlative sparsity patterns within these quantified constraints.

\vskip 6pt
\noindent{\bf Theorem B. (Theorem \ref{th::sparse}) }{\itshape
    Suppose that $\Y$ is compact, Assumption \ref{assump_CS} holds for $f\in\RR[\bx]$ and $\X$, 
    and Assumption \ref{assump2} holds with respect to each $\bx(\mI_\ell)$. 
    If $f>0$ on $\X$, then there exist SOS polynomials $\sigma_{\ell,0}\in\RR[\bx(\mI_\ell)]$, and SOS polynomial matrices $S_j\in\bS[\bx(\mI_\ell),\by]^{q_j}$, $j\in\mJ_{\ell}$, 
    such that
    \[f(\bx)=\sum_{\ell=1}^t\left(\sigma_{\ell,0}(\bx(\mI_\ell))+\sum_{j\in\mJ_{\ell}}\int_{\Y}\langle S_j(\bx(\mI_\ell),\by),G_j(\bx(\mI_\ell),\by)\rangle\ud\nu(\by)\right).\]
}
\vskip 6pt

% Note that the Archimedean condition in Theorems \ref{th::SHpsatz} require the
% set $\X$ to be compact. 
Next, building on Theorem A and existing techniques, 
we establish two generalized Positivstellens{\"a}tze for polynomial matrices with UQs 
and without assuming the Archimedean condition. 
The first is a Positivstellensatz for non-compact case. 

\vskip 6pt
\noindent{\bf Theorem C. (Theorem \ref{th::homo}) }{\itshape
Suppose that Assumption \ref{assump::carleman} holds,
$F\in\bS[\bx]^p$ and $G\in\bS[\bx, \by]^{q}$ are homogeneous in $\bx$ of even degree, and 
$F(\bx)\succ 0$ for all $\bx\in\X\setminus\{\mathbf{0}\}$. Then, there exists $N\in\N$ and 
SOS matrices $S_0\in\bS[\bx]^p$, $S\in\bS[\bx,\by]^{pq}$ which 
are homogeneous in $\bx$ and $\deg S_0=2N+\deg F$, $\deg_{\bx} S=2N+\deg F-\deg_{\bx}G$, 
such that 
\[
\Vert\bx\Vert^{2N}F(\bx)=S_0(\bx)+\int_{\Y}\lram{S(\bx,\by)}{G(\bx, \by)}{p}
 \ud\nu(\by).
 \]
}

% As a consequence of Theorem \ref{th::homo}, its inhomogeneous counterpart concerning the positive 
% semidefiniteness of $F(\bx)$  over $\X$ is provided in Corollary \ref{cor::inhomo}.

% Then, we present a P{\'o}lya-type Positivstellensatz in a homogeneous setting, 
% for polynomial matrices that are PD on the intersection of $\X$ with the nonnegative orthant $\RR^n_+$. 
%We now prove the following P{\'o}lya-type Positivstellensatze with universal quantifiers. 

%The second such Positivstellensatz is of P{\'o}lya-type. 
The second is a Positivstellensatz on the nonnegative orthant.
\vskip 6pt
\noindent{\bf Theorem D. (Theorem \ref{th::PolyaMPU}) }{\itshape
Suppose that 
$F\in\bS[\bx]^p$ and $G\in\bS[\bx, \by]^{q}$ are homogeneous in $\bx$, $F(\bx)\succ 0$ for all $\bx\in\RR^n_+\cap\X\setminus\{\mathbf{0}\}$, and $\Y$ is compact. 
Then, there exists $N\in\N$ and 
polynomial matrices $P_0=\sum_{\ba}P_{0,\ba}\bx^{\ba}\in\bS[\bx]^p$, 
$P=\sum_{\ba}P_{\ba}(\by)\bx^{\ba}\in\bS[\bx,\by]^{pq}$ which 
are homogeneous in $\bx$ and satisfy each $P_{0,\ba}\succ 0$, $P_{\ba}(\by)\succ 0$
for all $\by\in\Y$, such that 
\[
\left(\sum_{i=1}^n x_i\right)^{N}F(\bx)=P_0(\bx)+\int_{\Y}\lram{P(\bx, \by)}{G(\bx, \by)}{p}
 \ud\nu(\by).
 \]
}

The inhomogeneous counterparts of Theorems C and D appear in Corollaries \ref{cor::inhomo} and \ref{cor::PolyaMPU}, respectively.  

As applications, leveraging the SOS-structured certificates provided by the established Positivstellens\"atze, we propose hierarchies of SDP relaxations for the robust PMI constrained optimization 
problem and analyze their convergence. 

\vskip 5pt
\vskip 7pt
The rest of this paper is organized as follows. In Section \ref{sec::pre}, we review some preliminary concepts. Section \ref{sec::archi} presents a matrix-valued Positivstellensatz for polynomial matrices with UQs under the Archimedean condition, along with a sparse version in the presence of correlative sparsity. In Section \ref{sec::nonarchi}, without relying on the Archimedean condition, we derive two generalized
Positivstellens\"atze for polynomial matrices with UQs. In Section \ref{sec::app}, we apply the established Positivstellens\"atze to robust PMI constrained optimization. For the sake of readability, some lengthy and technical proofs are deferred to Section \ref{sec::proofs}. Conclusions are given in Section \ref{sec::cons}.

\section{Preliminaries}\label{sec::pre}
We collect some notation and basic concepts
which will be used in this paper. We denote by $\bx$ (resp., $\by$)
the $n$-tuple (resp., $m$-tuple) of variables $(x_1,\ldots,x_n)$ (resp.,
$(y_1,\ldots,y_m)$).
The symbol $\N$ (resp., $\RR$, $\RR_+$) denotes
the set of nonnegative integers (resp., real numbers, nonnegative real numbers). 
For a positive integer $n\in\N$, denote by $[n]$ the set $\{1,\ldots,n\}$.
Denote by $\RR^p$ (resp. $\RR^{l_1\times l_2}$, $\bS^{p}$, $\bS_+^{p}$) 
the space of $p$-dimensional real vector (resp. $l_1\times l_2$ real matrix, $p\times p$ symmetric real matrix,
$p\times p$ PSD matrix). 
Denote by $\RR^n_+$ the nonnegative orthant of $\RR^n$.
For $\bv\in\RR^p$ (resp., $N\in\RR^{l_1\times l_2}$), the symbol 
$\bv^{\intercal}$ (resp., $N^\intercal$) denotes the transpose of $\bv$ (resp., $N$). 
For a matrix $N\in\RR^{p\times p}$, $\tr{N}$ denotes its trace. 
For two matrices $N_1$ and $N_2$, $N_1\otimes N_2$ denotes the 
Kronecker product of $N_1$ and $N_2$.
For two matrices $N_1$ and $N_2$ of the same size, $\langle N_1, N_2\rangle$ denotes the 
inner product $\tr{N_1^{\intercal}N_2}$ of $N_1$ and $N_2$.
The notation $I_p$ denotes the $p\times p$ identity matrix.
%and $0_m$ denotes the  $m$-dimensional vector of all zeros.
For any $t\in \RR$, $\lceil t\rceil$ (resp., $\lfloor t\rfloor$) denotes the smallest (resp., largest)
integer that is not smaller (resp., larger) than $t$. For $\bu\in \RR^n$,
$\Vert \bu\Vert$ denotes the standard Euclidean norm of $\bu$.
For $N\in\RR^{l_1\times l_2}$, $\Vert N\Vert$ denotes the spectral norm of $N$.
For a vector $\ba=(\alpha_1,\ldots,\alpha_n)\in\N^n$,
let $\vert\ba\vert=\alpha_1+\cdots+\alpha_n$. For a set $A$, we use $\vert A\vert$ to denote its cardinality.
For $k\in\N$, let $\N^n_k\coloneqq\{\ba\in\N^n\mid \vert\ba\vert\le k\}$
and $\snd{k}=\binom{n+k}{k}$ be its cardinality.
For variables $\bx \in \RR^n$ and $\ba\in\N^n$, $\bx^{\ba}$ denotes the monomial
$x_1^{\alpha_1}\cdots x_n^{\alpha_n}$.
Let $\RR[\bx]$ (resp. $\bS[\bx]^p$) denote 
the set of real polynomials (resp. $p\times p$ symmetric real polynomial matrices) in $\bx$.
For $h\in\RR[\bx]$ (resp., $h\in\RR[\bx,\by]$), we denote by $\deg h$ (resp., $\deg_{\bx} h$) 
its total degree in $\bx$.
For a polynomial matrix $T(\bx)=[T_{ij}(\bx)]$ (resp., $T(\bx, \by)=[T_{ij}(\bx,\by)]$), denote $\deg T\coloneqq\max_{i,j}\deg T_{ij}$ (resp., $\deg_{\bx}T\coloneqq\max_{i,j}\deg_{\bx}T_{ij}$).
%\[
%\deg(T)\coloneqq\max\,\{\deg(T_{ij}) \mid i=1,\ldots,l_1,j=1,\ldots,l_2\}.
%\]
%For $h\in\RR[\bx]$, we denote by $\nabla_{\bx}(h)$ its gradient vector and 
%by $\nabla_{\bx\bx}(h)$ its Hessian matrix.
For $k\in\N$, denote by $\RR[\bx]_k$ (resp., $\bS[\bx]^p_k$) the subset of $\RR[\bx]$ (resp., $\bS[\bx]^p$)
of degree up to $k$.
For any $P(\bx)=[P_{ij}(\bx)]\in\bS[\bx]^p$ and $Q(\bx,\by)=[Q_{ij}(\bx, \by)]\in\bS[\bx,\by]^q$,
denote 
\[
\begin{aligned}
\supp(P)&\coloneqq\{\ba\in\N^n \mid \bx^{\ba}\ \text{appears in some }P_{ij}(\bx)\},\\
\suppx{\bx}{Q}&\coloneqq\{\ba\in\N^n \mid \bx^{\ba}\by^{\bb}\ \text{appears in some }Q_{ij}(\bx, \by) \text{ for some }\bb\in\N^m\},\\
\suppx{\by}{Q}&\coloneqq\{\bb\in\N^m \mid \bx^{\ba}\by^{\bb}\ \text{appears in some }Q_{ij}(\bx, \by) \text{ for some }\ba\in\N^n\}.
\end{aligned}
\]
%For a $\RR$-vector space $A$, denote by $A^*$ the dual space of 
%linear functionals from $A$ to $\RR$. Given a cone $B\subseteq A$, 
%its dual cone is $B^*\coloneqq\{L\in A^*\mid L(b)\ge 0, \ \forall b\in B\}$.

%\subsection{SOS matrices and positivstellens\"atz for polynomial matrices}\label{subsec::pre}
Let us recall some background on SOS matrices and positivstellens\"atz for polynomial matrices. 
For a polynomial $f(\bx)\in\RR[\bx]$, if there exist polynomials $f_1(\bx),\ldots,f_t(\bx)$ such that $f(\bx)=\sum_{i=1}^tf_i(\bx)^2$,
then we call $f(\bx)$ an SOS.
A polynomial matrix $S(\bx)\in\bS[\bx]^{p}$ is said to be an \emph{SOS matrix} if there exists an $l\times p$ polynomial matrix $T(\bx)$ for some
$l\in\N$ such that $S(\bx)=T(\bx)^{\intercal}T(\bx)$. For $d\in\N$, denote by $[\bx]_d$
the canonical basis of $\RR[\bx]_d$, i.e.,
\begin{equation}\label{eq::ud}
	[\bx]_d\coloneqq[1,\ x_1,\ x_2,\ \cdots,\ x_n,\ x_1^2,\ x_1x_2,\ \cdots,\
	x_n^d]^{\intercal},
\end{equation}
whose cardinality is $\snd{d}=\binom{n+d}{d}$. 
With $d=\deg T$, we can write $T(\bx)$ as
\[
	T(\bx)=Q([\bx]_d\otimes I_p) \text{ with } 
 Q=[Q_1,\ldots,Q_{\snd{d}}],\quad Q_i\in\RR^{l\times p},
\]
where $Q$ is the vector of the coefficient matrices of $T(\bx)$ with respect to
$[\bx]_d$. Hence, $S(\bx)$ is an SOS matrix with respect to $[\bx]_d$ if there
exists some $Q\in\RR^{l\times p\snd{d}}$ satisfying 
\[
	S(\bx)=T(\bx)^{\intercal}T(\bx)=([\bx]_d\otimes I_p)^{\intercal}(Q^{\intercal}Q)([\bx]_d\otimes I_p). 
\]
Thus, we have the following results.
\begin{proposition}{\upshape \cite[Lemma 1]{SH2006}}\label{prop::SOSrep}
A polynomial matrix $S(\bx)\in\bS[\bx]^{p}$ is an SOS matrix with
respect to the monomial basis $[\bx]_d$ if and only if there exists $Z\in\mathbb{S}_+^{p\snd{d}}$ such that  
$S(\bx)=([\bx]_d\otimes I_p)^{\intercal} Z ([\bx]_d\otimes I_p)$.
%For $i,j=1,\ldots,m$, the $(i,j)$-th block $Z_{ij}$ of $Z$ with respect to
%the $(s_d,\ldots,s_d)$-partitioning satisfying 
%\[
%u_d(\bx)^TZ_{ij}u_d(\bx)=S_{ij}(\bx).
%\]
\end{proposition}

\begin{lemma}\label{prop::submatrix}
Let $S(\bx)\in \bS[\bx]^{p}$ be an SOS matrix and $S_k(\bx)\in\bS[\bx]^{k\times k}$ 
be a principal submatrix of $S(\bx)$ whose rows and columns are indexed by $(p_1,\ldots,p_k)$ with
$1\le p_1<\cdots<p_k\le p$, then $S_k(\bx)$ is an SOS matrix. 
\end{lemma}
\begin{proof}
As $S(\bx)$ is an SOS matrix, there exists an $l\times p$ polynomial matrix $T(\bx)$ for some
$l\in\N$ such that $S(\bx)=T(\bx)^{\intercal}T(\bx)$. Denote
by $T_k(\bx)$ the submatrix of $T(\bx)$ consisting of the columns of
$T(x)$indexed by $(p_1,\ldots,p_k)$. Then $S_k(\bx)=T_k(\bx)^{\intercal}T_k(\bx)$ and hence is an SOS matrix.
\end{proof}

We next recall Scherer-Hol's Positivstellensatz for polynomial matrices obtained in \cite{SH2006}.
Define the bilinear mapping 
\[
\lram{\cdot}{\cdot}{p} \colon \RR^{pq\times pq}\times \RR^{q\times q} \to \RR^{p\times p},\quad
\lram{A}{B}{p}=\trp{A^{\intercal}(I_p\otimes B)}, 
\]
with
\[
\trp{C}\coloneqq\left[
\begin{array}{ccc}
\tr{C_{11}} & \cdots & \tr{C_{1p}}\\
\vdots & \ddots & \vdots\\
\tr{C_{p1}} & \cdots & \tr{C_{pp}}
\end{array}
\right]
\quad\text{for } C=[C_{ij}]_{i,j\in[p]}\in\RR^{pq\times pq}, C_{ij}\in\RR^{q\times q}.
\]
When $p=1$, the product $\lram{A}{B}{1}$ coincides with the matrix
inner product $\langle A, B\rangle=\tr{A^\intercal B}$.

Let $\mathbf{H}\coloneqq\{H_1,\ldots,H_t\}$ where each $H_j\in\bS[\bx]^{r_j}$ for some $r_j\in\N$.
\begin{lemma}\label{prop::block}
Let $H(\bx)=\diag{H_1(\bx),\ldots,H_t(\bx)}$ be a block diagonal matrix, then for any SOS matrix $S(\bx)\in\bS[\bx]^{pr}$ where $p\in\N$ and $r=r_1+\cdots+r_t$, 
there are SOS matrices $S_j(\bx)\in\bS[\bx]^{r_j}$, $j\in[t]$,
such that 
\[
	\lram{S(\bx)}{H(\bx)}{p}=\sum_{j=1}^t \lram{S_j(\bx)}{H_j(\bx)}{p}.
\]
\end{lemma}
\begin{proof}
   For each $j\in[t]$, let $S_j(\bx)$ be the $pr_j\times pr_j$ principal submatrix of $S(\bx)$ whose rows and columns are indexed by 
   \[
   \left(\sum_{\ell=1}^{j-1}r_{\ell}+1, \sum_{\ell=1}^{j}r_{\ell}, 
   r+\sum_{\ell=1}^{j-1}r_{\ell}+1, r+\sum_{\ell=1}^{j}r_{\ell}, \ldots,
   (p-1)r+\sum_{\ell=1}^{j-1}r_{\ell}+1, (p-1)r+\sum_{\ell=1}^{j}r_{\ell}\right).
   \]
   As $H(\bx)$ is block diagonal, by the definition of the mapping $\lram{.}{.}{p}$, 
   it is easy to see that 
\[
\lram{S(\bx)}{H(\bx)}{p}=\sum_{j=1}^t \lram{S_j(\bx)}{H_j(\bx)}{p}.
\]
By Lemma \ref{prop::submatrix}, each $S_j(\bx)$ is an SOS matrix.
\end{proof}

%\begin{remark}\label{rk::mp}
%Note that $(A, B)_1$ is just the standard inner product 
%$\langle A, B\rangle=\tr{A^{\intercal}B}$. Moreover, we have $(A, B)_m\succeq 0$ if $A\succeq 0$
%and $B\succeq 0$ (\cite{SH2006}).
%\end{remark}

The \emph{matrix quadratic module} $\QM^p(\mathbf{H})$ generated by $\mathbf{H}$ is defined as the cone
\[
	\QM^p(\bH)\coloneqq\left\{S_0(\bx)+\sum_{j=1}^t\lram{S_j(\bx)}{H_j(\bx)}{p}\ 
 \middle\vert \
S_0\in\bS[\bx]^{p}, S_j\in\bS[\bx]^{pr_j}, j\in[t],\ \text{are SOS}
\right\}.
\]
The $k$-th truncated matrix quadratic module $\QM^p_k(\mathbf{H})$ is a subcone of $\QM^p(\mathbf{H})$ in which the degrees of the polynomial matrices $S_0(\bx)$ and $\sum_{j=1}^t\lram{S_j(\bx)}{H_j(\bx)}{p}$ are bounded by $2k$.
%
%For each $k\in\N$, we define the $k$-th \emph{truncated matrix quadratic module} $\QM^p_k(\mathbf{H})$ 
%associated with $\mathbf{H}$ by 
%\[
%	\QM^p_k(\bH)\coloneqq\left\{S_0(\bx)+\lram{S_j(\bx)}{H_j(\bx)}{p}\ 
% \middle\vert \
%\begin{aligned}
%&S_0\in\bS[\bx]^{p}, S_j\in\bS[\bx]^{pr_j},\ \text{are SOS},\\
%&\deg(S_0), \deg\lram{S_j}{H_j}{p}\le 2k
%\end{aligned}
%\right\},
%\]
%and define the \emph{matrix quadratic module} by $\QM^p(\bH)\coloneqq\bigcup_{k\in\N}\QM^p_k(\bH).$
%By Proposition \ref{prop::SOSrep}, checking membership in $\QM^p_k(\bH)$ can be accomplished with an SDP.
\begin{assumption}\label{assump_arch}
{\rm
%The quadratic module 
$\QM^p(\bH)$ is {\itshape Archimedean}, i.e.,
there is $C>0$ such that $C - \Vert \bx\Vert^2\in\QM^1(\bH)$.}
\end{assumption}
\begin{theorem}{\upshape (Scherer-Hol's Positivestellensatz)}\label{th::psatz}
Let Assumption \ref{assump_arch} 
and $F(\bx)\in\bS[\bx]^{p}$ be PD on 
$\{\bx\in\RR^n \mid H_j(\bx)\succeq 0, \ j\in[t]\}$. Then $F(\bx)\in\QM^p(\bH)$.
% In particular, if $G(\bx)=\diag{g_1(\bx),\ldots,g_q(\bx)}$ is diagonal,  there are
% $p\times p$ SOS polynomial matrices $S_0(\bx), S_1(\bx), \ldots,
% S_q(\bx)
% \in\RR[\bx]^{p\times p}$ such that 
%	\[
%		F(\bx)=S_0(\bx)+\sum_{i=1}^q g_i(\bx)S_i(\bx).
%	\]
\end{theorem}
\begin{proof}
It was proved in  \cite[Corollary 1]{SH2006} for the case $t=1$.
The case $t>1$ can be derived from \cite[Corollary 1]{SH2006} and Lemma \ref{prop::block}.
\end{proof}

%In particular, if $G(\bx)=\diag{g_1(\bx),\ldots,g_q(\bx)}$,  it holds
%\[
%\QM^m_k(G)\coloneqq\left\{S_0(\bx)+\sum_{i=1}^q g_i(\bx)S_i(\bx)\ \middle\vert\ 
%\begin{aligned}
%&S_i\in\bS[\bx]^{m}\ \text{is SOS},\\
%&\deg(S_0), \deg(g_iS_i)\le 2k
%\end{aligned}
%\right\}.
%\]
% If $m=1$, we use the notation $\QM(G)$ (resp., $\QM_k(G)$) instead of $\QM^1(G)$ (resp., $\QM^1_k(G)$) for simplicity.
% \begin{remark}
%    For a set of polynomials $H(\bx)=\{h_1(\bx),\ldots,h_s(\bx)\}\subset\RR[\bx]$,
%    we also use the notation $\QM^m(H)$ (resp., $\QM^m_k(H)$) denote the 
%    quadratic module ($k$-th quadratic module) associated with the diagonal
%    polynomial matrix $\diag{h_1(\bx),\ldots,h_s(\bx)}$. In particular, when $m=1$, 
%    $\QM(H)$ (resp., $\QM_k(H)$) becomes the standard quadratic module 
%    ($k$-th quadratic module) associated with $H$, and Theorem \ref{th::psatz}
%    becomes Putinar's Positivstellensatz for the basic semi-algebraic set
%    \[
%    \{\bx\in\RR^n \mid h_1(\bx)\ge 0,\ldots,h_s(\bx)\ge 0\}. 
%    \]
%    See \cite{Putinar1993} for more details.
% \end{remark}

\section{Positivstellens\"atze for polynomial matrices with UQs under the Archimedean condition}\label{sec::archi}
In this section, assuming the Archimedean condition,
we shall present a matrix-valued Positivstellensatz incorporating universal quantifiers,  
along with a sparse version in the presence of correlative sparsity.
%Let $\bG\coloneqq\{G_1,\ldots,G_s\}$.
%Fo each $k\in\N$, we define 
%$k$-th \emph{truncated matrix quadratic module} 
%$\QM^p_k(G, \nu)$ generated by $G$ and $\nu$ by 
%\[
%	\QM^p_k(G,
%	\nu)\coloneqq\left\{S_0+\int_{\Y}\lram{S}{G}{p}\ud\nu(\by)\ \middle\vert \
%\begin{aligned}
%&S_0\in\bS[\bx]^{p}, S\in\bS[\bx, \by]^{pq} \text{are SOS},\\
%%&S_0, S\ \text{are SOS matrices},\\
%&\deg(S_0), \deg\lram{S}{G}{p}\le 2k
%\end{aligned}
%\right\},
%\]
%and define the \emph{matrix quadratic module} by 
%\[
%\QM^p(G, \nu)\coloneqq\bigcup_{k\in\N}\QM^p_k(G, \nu).
%\]

\subsection{A Positivstellensatze for polynomial matrices with UQs}
%Recall the definition of the matrix quadratic module $\QM^p(G, \nu)$ generated by $G$ and $\nu$.
Recall the set $\X$ in \eqref{defineX}.
Throughout the paper, let $\nu$ be a fixed Borel measure on $\RR^m$ with support $\supp(\nu)=\Y$ 
and satisfying $\int_{\Y}|h(\by)|\ud\nu(\by)<\infty$ for all $h(\by)\in\RR[\by]$.
Similarly to the scalar case in \cite{HKN2024}, let us define the \emph{matrix quadratic module}
associated with $(G, \nu)$ as follows.
\begin{definition}\label{def::MQM}{\rm
The matrix quadratic module $\QM^p(G, \nu)$ generated by $G$ and $\nu$ is defined as
\[
	\QM^p(G,
	\nu)\coloneqq\left\{S_0+\int_{\Y}\lram{S}{G}{p}\ud\nu(\by)\ \middle\vert \
S_0\in\bS[\bx]^{p}, S\in\bS[\bx, \by]^{pq} \,\,\text{are SOS matrices}
\right\}.
\]   }
\end{definition}

Next, we derive a Positivstellensatze for polynomial matrices with UQs, providing SOS-structured 
characterizations for polynomial matrices that are PD over $\X$.
\begin{assumption}\label{assump2}
{\rm
%The quadratic module 
$\QM^p(G, \nu)$ is {\itshape Archimedean}, i.e.,
there is $C>0$ such that $C - \Vert \bx\Vert^2\in\QM^1(G, \nu)$.}
\end{assumption}

Consider the 
%following multivariable 
Carleman condition imposed on $\nu$ which is 
automatically satisfied when $\Y$ is compact.
\begin{assumption}\label{assump::carleman}
	The Borel measure $\nu$ satisfies the multivariate Carleman condition
	\[
		\sum_{d=1}^\infty \left(\int_{\Y}
		y_j^{2d}\ud\nu(\by)\right)^{-\frac{1}{2d}}=\infty,\quad \forall j\in
		[m].
	\]
\end{assumption}
\begin{proposition}\cite[Proposition 3.2]{HKN2024}\label{prop::dense}
	Suppose that $\nu$ satisfies Assumption \ref{assump::carleman}.
	Then, SOS polynomials are dense in the cone of nonnegative
	functions in $L^2(\RR^m, \nu)$. 
\end{proposition}

\begin{proposition}\label{prop::seteq}
    Suppose that $\nu$ satisfies Assumption \ref{assump::carleman}. Then  
    \begin{equation}\label{eq::seteq}
    \X=\{\bx\in\RR^n \mid H(\bx)\succeq 0,\ \forall H\in \QM^p(G, \nu)\}.
    \end{equation}
\end{proposition}
\begin{proof}
    By the definition of $\QM^p(G, \nu)$, we only need to prove that 
    $\X$ contains the set on the right-hand side of the equation in \eqref{eq::seteq}.
 Fix $\bu\in\RR^n$ with $H(\bu)\succeq 0$ for all $H\in \QM^p(G, \nu)$. 
Suppose on the contrary that $\bu\not\in\X$,
 i.e., there exists $\bw\in\Y$ such that $G(\bu, \bw)\not\succeq 0$. Then, there is a ball
   $\mathcal{O}\subset\RR^m$ with a radius $\rho>0$ around $\bw$ such that
   $G(\bu, \by)\not\succeq 0$ on $2\mathcal{O}$.
   We may assume that there exists $\bv\in\RR^{q}$ and $\delta>0$ such that 
    $\bv^{\intercal}G(\bu, \by)\bv\le -\delta$ on $2\mathcal{O}$.
    Define a continuous function $h(\by)$ on $\RR^m$ by $h(\by)=2\rho-\Vert\by-\bw\Vert$
    for $\by\in2\mathcal{O}$ and $h(\by)=0$ otherwise. 
    By Proposition~\ref{prop::dense}, there exists a sequence of SOS polynomials 
    $\{\sigma_k\}_k$ in $\RR[\by]$ that converges to $h$ in the $L^2$-norm.
%    Then, $\bv^{\intercal}G_{j_0}\bv h_k^2$ converges to $\bv^{\intercal}G_{j_0}\bv h^2$ uniformly 
%	on $\mathcal{C}\times\Y$ and 
    Hence, 
\[
    \begin{aligned}
		&\lim_{k\to\infty}\int_{\Y}\bv^{\intercal}G(\bu, \by)\bv \sigma_k(\by)\ud\nu(\by)
  =\int_{\Y}\bv^{\intercal}G(\bu, \by)\bv h(\by)\ud\nu(\by)\\
        =&\int_{\Y\cap 2\mathcal{O}}\bv^{\intercal}G(\bu,
		\by)\bv (2\rho-\Vert\by-\bw\Vert)\ud\nu(\by)
		\le \int_{\Y\cap
			\mathcal{O}}-\delta\rho\ud\nu(\by)
	= -\delta\rho\nu(\Y\cap\mathcal{O})<0,
    \end{aligned}
\]
    where the last inequality is due to the fact that $\supp(\nu)=\Y$ and thus $\nu(\Y\cap\mathcal{O})>0$. 
%where the first equality 
%in \eqref{eq::ineq1} can be verified by the convergence of $\{\sigma_k\}_k$ to $h^{(2)}$ in the $L^2$-norm. 
Note that for all $k\in\N$,
\[
		\int_{\Y}\bv^{\intercal}G(\bu, \by)\bv \sigma_k(\by)\ud\nu(\by)
     =\frac{1}{p}\tr{\int_{\Y}\lram{\sigma_k(\by)I_p\otimes\bv\bv^{\intercal}}{G(\bu, \by)}{p}\ud\nu(\by)}\ge 0,
\]
since $\sigma_k(\by) I_p\otimes\bv\bv^{\intercal}$ is an SOS and $H(\bu)\succeq 0$ for all $H\in \QM^p(G, \nu)$.
A contradiction follows.
\end{proof}

Now, we present our main result concerning a Positivstellensatz for polynomial matrices with UQs.
\begin{theorem}\label{th::SHpsatz}
Suppose that Assumptions \ref{assump2} and \ref{assump::carleman} hold. 
If $F(\bx)\in\bS[\bx]^{p}$ is PD on $\X$, then $F(\bx)\in\QM^p(G, \nu)$.
\end{theorem}
\begin{proof}
Since the quadratic module $\QM^p(G, \nu)$ is Archimedean and the equality in \eqref{eq::seteq} holds,
the conclusion follows from the fundamental Positivstellensatz for matrix algebras of polynomials
\cite[Theorem 10.25]{Sch2020}.
%The proof is by contradiction, using the functional $\mL$ from Corollary \ref{cor::sl},  
%and proceeds as follows:
%\begin{enumerate}[1. ]
%    \item Construct a Hilbert space $\mH$ and operators $\widehat{\bA}\coloneqq(\widehat{A}_i)_{i\in [n]}$ 
%    with $\widehat{A}_i\in\mathcal{B}(\mH)$, where $\mathcal{B}(\mH)$ denotes the set of bounded linear operators on $\mH$. Then construct a state $\Pi$ on the
%    the $p\times p$ matrix algebra $\mathcal{M}_p(\mathcal{B}(\mH))$ over $\mathcal{B}(\mH)$ such that $\Pi(F(\widehat{\bA}))=\mL(F)$;
%    \item By applying the GNS construction to $(\mathcal{M}_p(\mathcal{B}(\mH)), \Pi)$, 
%    construct a Hilbert space $\mathcal{K}_1$ with inner product $\langle\cdot,\cdot\rangle_{\mK_1}$ and 
%    operators $\bA=(A_i)_{i\in [n]}$ with $A_i\in\mathcal{B}(\mathcal{K}_1)$ such that 
%    $\mL(F)=\sum_{i,j=1}^p\left\langle F_{ij}(\bA)\eta_j, \eta_i\right\rangle_{\mK_1}$ 
%    for some $\eta_i\in\mathcal{K}_1$, $i\in [p]$;
%    \item By applying the spectral theorem to $\bA=(A_i)_{i\in [n]}$, 
%    construct a $p\times p$ PSD matrix-valued measure $\Phi$ on the Borel $\sigma$-algebra $\mathfrak{B}(\RR^n)$ such that $\mL(F)=\int_{\RR^n}F(\bx)\ud\Phi(\bx)$;
%    \item Show that the support of $\Phi$ is contained in $\X$ and derive a contradiction by Corollary \ref{cor::sl}.
%\end{enumerate}
%For the sake of readability, the details of the proof is postponed to Section \ref{sec::prfThm3.1}.
\end{proof}

Next, we derive a corollary of Theorem \ref{th::SHpsatz}, 
which will be used in Section \ref{sec::nonarchi}. 
%We need the following lemma.
\begin{lemma}\label{lem::int}
    Suppose that $S(\bx,\by)\in\bS[\bx,\by]^q$ is an SOS matrix in $\bx$ and $\by$. 
    Then $\int_{\Y}S(\bx,\by)\ud\nu(\by)$ is an SOS matrix in $\bx$.
\end{lemma}
\begin{proof}
  As $S(\bx,\by)$ is an SOS matrix, there exists an $\ell\times q$ polynomial matrix $T(\bx,\by)$ 
  for some $\ell\in\N$ such that $S(\bx,\by)=T(\bx,\by)^{\intercal}T(\bx,\by)$.  
  With $d=\deg_{\bx} T$, we could write $T(\bx,\by)$ as
\[
	T(\bx,\by)=Q(\by)([\bx]_d\otimes I_q) \text{ with } 
 Q(\by)=[Q_1(\by),\ldots,Q_{\snd{d}}(\by)],\quad Q_i(\by)\in\RR[\by]^{\ell\times q},
\]
where $Q(\by)$ is the vector of coefficient matrices of $T(\bx,\by)$ (considered as a polynomial
matrix in $\RR[\bx]^{\ell\times q}$) with respect to $[\bx]_d$. Hence,
\[
\begin{aligned}
\int_{\Y}S(\bx,\by)\ud\nu(\by)
=\int_{\Y}([\bx]_d\otimes I_q)^{\intercal}(Q(\by)^{\intercal}Q(\by))([\bx]_d\otimes I_q)\ud\nu(\by)=([\bx]_d\otimes I_q)^{\intercal}\int_{\Y}Q(\by)^{\intercal}Q(\by)\ud\nu(\by)([\bx]_d\otimes I_q)
\end{aligned}
\]
As $\int_{\Y}Q(\by)^{\intercal}Q(\by)\ud\nu(\by)$ is PSD, 
$\int_{\Y}S(\bx,\by)\ud\nu(\by)$ is an SOS matrix in $\bx$.
\end{proof}

\begin{corollary}\label{cor::pos}
Let $\mathbf{H}\coloneqq\{H_1,\ldots,H_t\}$ where each $H_j\in\bS[\bx]^{r_j}$ for some $r_j\in\N$.
    Suppose that Assumption~\ref{assump::carleman} holds and there exists $C>0$ such that 
    $C - \Vert \bx\Vert^2\in\QM^1(G, \nu)+\QM^1(\bH)$.
If $F(\bx)\in\bS[\bx]^p$ is PD on $\X\cap \{\bx\in\RR^n \mid H_j(\bx)\succeq 0, \ j\in[t]\}$, 
then $F(\bx)\in\QM^p(G, \nu)+\QM^p(\bH)$.
\end{corollary}
\begin{proof}
It is clear that $F(\bx)\succ 0$ on 
$\X\cap \{\bx\in\RR^n \mid H_j(\bx)\succeq 0, \ j\in[t]\}$ if and only if 
$F(\bx)\succ 0$ on 
\[
\left\{\bx\in\RR^n \,\middle|\, \widehat{G}(\bx,\by)\coloneqq\diag{G(\bx, \by), H_1(\bx),\ldots,H_t(\bx)}\succeq 0, 
\ \forall \by\in\Y\right\}.
\]
By Lemmas \ref{prop::block} and \ref{lem::int}, it is easy to see that 
$\QM^p(\widehat{G}, \nu)=\QM^p(G, \nu)+\QM^p(\bH)$.
So, the conclusion follows from Theorem \ref{th::SHpsatz}.
\end{proof}

\subsection{A sparse Positivstellensatz for polynomial matrices with UQs}
There are sparse Positivstellens\"atze for scalar polynomials being positive 
on a basic semialgebraic set in the presence of correlative sparsity; see \cite{GNS2007,LasserreSparsity} where the set is defined by polynomial inequalities 
and \cite{KM2009} where the set is defined PMIs. 
Let $p=1$. Building on the ideas from \cite{GNS2007} and \cite{KM2009}, we now prove a sparse Positivstellensatz for polynomial matrices with UQs.
Let $f\in\RR[\bx]$ and
\begin{equation}\label{def::sparseX}
\widehat{\X}\coloneqq\{\bx\in\RR^n \mid G_j(\bx, \by)\succeq 0, \ \forall j\in[s],\ \by\in \Y\subset\RR^m\},
\end{equation}
where each $G_j\in\bS[\bx, \by]^{q_j}$, $q_j\in\N$. 

\begin{assumption}[correlative sparsity pattern]\label{assump_CS}
{\rm
   Subsets $\{\mI_\ell\}_{\ell\in[t]}$ of $[n]$ and subsets $\{\mJ_\ell\}_{\ell\in[t]}$ of $[s]$ satisfy the following conditions:
   \begin{enumerate}[\upshape (i)]
       \item The {\itshape running intersection property} holds for $\{\mI_\ell\}_{\ell\in[t]}$, i.e.,
       \[
\text{for } \ell=2,\ldots,t,\ \exists k<\ell\quad\text{s.t.}\quad \mathcal{I}_{\ell}\cap \bigcup_{j<\ell}
\mathcal{I}_j\subseteq \mI_k;
\]
       \item For every $\ell\in[t]$ and $j\in\mJ_{\ell}$, $G_j\in\bS[\bx(\mI_{\ell}), \by]^{q_j}$, 
       where $\bx(\mI_{\ell})\coloneqq\{x_i\}_{ i\in\mI_{\ell}}$;
       \item $f$ decomposes as $f=f_1+\cdots+f_t$ with each $f_\ell\in\RR[\bx(\mI_\ell)]$.
   \end{enumerate}
   }
\end{assumption}

%\revision{In particular, for the case $p=1$, we have
%\begin{corollary}\label{cor::f}
%Suppose $f(\bx)>0$ on $\widehat{\X}$. Then for any $C>0$, there exists
%$M>0$ and $\bar{k}\in\N$ such that 
%\begin{equation}\label{eq::ge2}
%   f(\bx)-\sum_{j=1}^s\int_{\Y}\left\langle\left(I_{q_j}-G_j(\bx, \by)/M\right)^{2k}, G_j(\bx, \by)\right\rangle\ud\nu(\by)>0  
%\end{equation}
%   on $[-C, C]^n$ for all $k\ge\bar{k}$. 
%\end{corollary}}

For each $\ell\in[t]$, let $\bG^{\ell}\coloneqq\{G_j\}_{ j\in\mJ_{\ell}}$ and $\QM^p(\bG^{\ell}, \nu)$
be the quadratic module generated by $\bG^{\ell}$ and $\nu$ in $\bS[\bx(\mI_{\ell})]^p$, i.e.,
\[
	\QM^p(\bG^{\ell},\nu)\coloneqq\left\{S_0+\sum_{j\in\mJ_{\ell}}\int_{\Y}\lram{S_j}{G_j}{p}\ud\nu(\by)\ \middle\vert \
S_0\in\bS[\bx(\mI_\ell)]^{p}, S_j\in\bS[\bx(\mI_\ell), \by]^{pq_j}, j\in\mJ_{\ell}, \text{ are SOS matrices}
\right\}.
\]

%By combining Propositions \ref{prop::main} with \ref{prop::decompf}, 
Using the correlative sparsity pattern and the Archimedean condition, 
we are able to derive the following sparse Positivstellensatz.

\begin{theorem}\label{th::sparse}
    Suppose that $\Y$ is compact, Assumption \ref{assump_CS} holds for $f$ and $\widehat{\X}$, 
    and Assumption \ref{assump2} holds for each $\QM^1(\bG^{\ell}, \nu)$, $\ell\in[t]$. 
    If $f>0$ on $\widehat{\X}$, then $f\in\sum_{\ell=1}^t\QM^1(\bG^{\ell}, \nu)$.
\end{theorem}

To prove Theorem \ref{th::sparse}, we need the following intermediary results.
\begin{proposition}\label{prop::decompf}\cite[Lemma 3]{GNS2007}
   Let $\{\mI_\ell\}_{\ell\in[t]}$ satisfy the running intersection property. For any $C>0$, if $f=f_1+\cdots+f_t$ with $f_\ell\in\RR[\bx(\mI_\ell)]$ satisfies $f>0$ on $[-C, C]^n$, then
   $f=h_1+\ldots+h_t$ for some $h_\ell\in\RR[\bx(\mI_\ell)]$ with $h_\ell>0$ on $[-C, C]^{|\mI_\ell|}$.
\end{proposition}

\begin{proposition}\label{prop::main} Suppose $F(\bx)\in\bS[\bx]^{p}$ is PD on $\widehat{\X}$. 
Then for any $C>0$, there exists $M>0$ and $\bar{k}\in\N$ such that 
\begin{equation}\label{eq::ge}
   F(\bx)-\sum_{j=1}^s\int_{\Y}\lram{I_p\otimes \left(I_{q_j}-G_j(\bx,
   \by)/M\right)^{2k}}{G_j(\bx, \by)}{p}\ud\nu(\by)\succ 0 
    \end{equation}
   on $[-C, C]^n$ for all $k\ge\bar{k}$. 
\end{proposition} 
%\begin{proof} 
%See Section \ref{sec::prfPropmain}.
%\end{proof}
The proof of Proposition \ref{prop::main} is postponed to Section \ref{sec::prfPropmain}.
Note that using Proposition \ref{prop::main}, we can provide a constructive proof of Theorem \ref{th::SHpsatz} assuming that $\Y$ is compact (see Section \ref{sec::prfPropmain}).

\begin{proof}[Proof of Theorem \ref{th::sparse}]
By assumption, let $C>0$ be such that $C - \Vert \bx(\mI_{\ell})\Vert^2\in\QM^1(\bG^{\ell}, \nu)$ 
holds for all $\ell\in[t]$. By Proposition \ref{prop::main}, there exists $M>0$ and $k'\in\N$ such that 
\[
   f(\bx)-\sum_{j=1}^s\int_{\Y}\left\langle\left(I_{q_j}-G_j(\bx, \by)/M\right)^{2k'}, G_j(\bx, \by)\right\rangle\ud\nu(\by)>0  
\]
   on $[-\sqrt{C}, \sqrt{C}]^n$. Note that for each $j\in[s]$, 
\[
\int_{\Y}\left\langle\left(I_{q_j}-G_j(\bx, \by)/M\right)^{2k'}, G_j(\bx, \by)\right\rangle\ud\nu(\by)\in
\RR[\bx(\mI_{\ell})]
\]
for some $\ell\in[t]$. By Proposition \ref{prop::decompf}, there exist
$h_\ell\in\RR[\bx(\mI_\ell)]$, $\ell\in[t]$,
such that
\[
f(\bx)-\sum_{j=1}^s\int_{\Y}\left\langle\left(I_{q_j}-G_j(\bx, \by)/M\right)^{2k'}, G_j(\bx, \by)\right\rangle\ud\nu(\by)=h_1+\cdots+h_t,
\]
and each $h_\ell>0$ on $[-\sqrt{C}, \sqrt{C}]^{|\mI_\ell|}$.
By Putinar's Positivstellensatz \cite{Putinar1993}, for each $h_{\ell}$, there exist SOS polynomials 
$\sigma_{\ell,0}, \sigma_{\ell,1}\in\RR[\bx(\mI_{\ell})]$ such that 
\[
h_{\ell}=\sigma_{\ell, 0}+\sigma_{\ell, 1}(C-\Vert\bx(\mI_{\ell})\Vert^2).
\]
As each $C - \Vert \bx(\mI_{\ell})\Vert^2\in\QM^1(\bG^{\ell}, \nu)$, it holds 
$h_{\ell}\in\QM^1(\bG^{\ell}, \nu)$. Therefore,
\[
f(\bx)=\sum_{\ell\in[t]}\left(\sum_{j\in\mJ_{\ell}}\int_{\Y}\left\langle\left(I_{q_j}-G_j(\bx, \by)/M\right)^{2k'}, G_j(\bx, \by)\right\rangle\ud\nu(\by)+h_{\ell}\right)
\in\sum_{\ell=1}^t\QM^1(\bG^{\ell}, \nu).
\]
\end{proof}

\begin{remark}
One might wonder whether the result of Theorem \ref{th::sparse} holds for a polynomial matrix $F(\bx)\in\bS[\bx]^p$
with $p>1$. 
Indeed, this is not true even in the absence of UQs; see \cite{MWG2024} for a counterexample.
\end{remark}

\section{Positivstellens\"atze for polynomial matrices with UQs and without the Archimedean condition}\label{sec::nonarchi}
The results in Section \ref{sec::archi} are derived under the Archimedean condition on the quadratic 
module generated by $(G, \nu)$, which requires $\X$ to be bounded. 
Many efforts have been made to provide SOS-structured 
representations for polynomials that are PD on
a semialgebraic set without the Archimedean condition. 
The goal of this section is to extend some of these well-known results to 
the setting of polynomial matrices with UQs. 

\subsection{A Positivstellensatz for the non-compact case}
Let $\theta\coloneqq1+\Vert\bx\Vert^2$.
For every nonnegative polynomial $f\in\RR[\bx]$ on a general basic semialgebraic set, 
Putinar and Vasilescu \cite{PV1999} proved that for a given $\varepsilon>0$ and $2d\ge \deg f$, 
there exists a nonnegative integer $k$ such that $\theta^k(f+\varepsilon \theta^d)$ belongs to the quadratic module associated with the basic semialgebraic set. 
Using Jacobi's technique \cite{Jacobi2001}, Mai et al. \cite{MLM2022}
provided an alternative proof of Putinar and Vasilescu's
result with an effective degree bound on polynomials involved in such certificates.

Building on Theorem \ref{th::SHpsatz} and similar techniques from \cite{MLM2022}, 
we next derive matrix-valued Positivstellens{\"a}tz for the non-compact case, 
incorporating universal quantifiers. In the homogeneous case, the result is stated as follows.
See Section \ref{sec::prfThmhom} for its proof.
\begin{theorem}\label{th::homo}
Suppose that Assumption \ref{assump::carleman} holds,
$F\in\bS[\bx]^p$ and $G\in\bS[\bx, \by]^{q}$ are homogeneous in $\bx$ of even degree, and 
$F(\bx)\succ 0$ for all $\bx\in\X\setminus\{\mathbf{0}\}$. Then, there exists $N\in\N$ and 
SOS matrices $S_0\in\bS[\bx]^p$, $S\in\bS[\bx,\by]^{pq}$ that 
are homogeneous in $\bx$ and $\deg S_0=2N+\deg F$, $\deg_{\bx}S=2N+\deg F-\deg_{\bx}G$, 
such that 
\[
\Vert\bx\Vert^{2N}F(\bx)=S_0(\bx)+\int_{\Y}\lram{S(\bx,\by)}{G(\bx, \by)}{p}
 \ud\nu(\by).
 \]
\end{theorem}

For any $H=[H_{ij}]\in\bS[\bx,\by]^r$, 
let $d_H\coloneqq\max\left\{\lfloor\deg_{\bx}(H_{ij})/2\rfloor+1:i, j\in[r]\right\}$
and $\td{\bx}\coloneqq(\bx, x_{n+1})$. 
By applying Theorem \ref{th::homo}, we obtain its inhomogeneous counterpart, as given below.
%
%For any $h(\bx)\in\RR[\bx]$, we denote $d_h=\lfloor\deg_{\bx}(h)/2\rfloor+1$ and 
%its $2d_h$-homogenization $\td{h}(\td{\bx})=x_{n+1}^{2d_h}h(\bx/x_{n+1})$.
\begin{corollary}\label{cor::inhomo}
Suppose that Assumption \ref{assump::carleman} holds and
%$\Y$ is compact and 
$F(\bx)\in\bS[\bx]^p$ is PSD on $\X$.   
Then for any $\varepsilon>0$, there exists $N_{\varepsilon}\in\N$,
SOS matrices $S_0\in\bS[\bx]^p$, and $S\in\bS[\bx,\by]^{pq}$ with $\deg S_0\le 2(N_{\varepsilon}+d_F)$ and $\deg_{\bx}S\le 2(N_{\varepsilon}+d_F-d_G)$, such that 
\[
\theta^{N_{\varepsilon}}(F(\bx)+\varepsilon\theta^{d_F}I_p)
=S_0(\bx)+\int_{\Y}\lram{S(\bx,\by)}{G(\bx, \by)}{p} \ud\nu(\by).
\]
\end{corollary}
\begin{proof}
Let $\wt{F}\coloneqq[x_{n+1}^{2d_F}F_{ij}(\bx/x_{n+1})]_{i,j\in[p]}\in\bS[\td{\bx}]^p$,
$\wt{G}\coloneqq[x_{n+1}^{2d_G}G_{ij}(\bx/x_{n+1}, \by)]_{i,j\in[q]}\in\bS[\td{\bx},\by]^{q}$,
%where 
%\[
%d_G\coloneqq\max\{\lfloor\deg_{\bx}(G_{ij})/2\rfloor+1,\ \ i, j\in[q]\},
%\]
and consider
\[
\wt{\X}\coloneqq\left\{\td{\bx}\in\RR^{n+1} \middle| \wt{G}(\td{\bx}, \by)\succeq 0, \ \forall \by\in \Y\right\}.
\]
We first prove that $\wt{F}+\varepsilon\Vert\td{\bx}\Vert^{2d_F}I_p\succ 0$ 
on $\wt{\X}\setminus\{\mathbf{0}\}$. Fix a point $\td{\bu}=(\bu, u_{n+1})\in\wt{\X}\setminus\{\mathbf{0}\}$.
\begin{enumerate}[-]
    \item Case 1: $u_{n+1}\neq 0$. As $\wt{G}(\td{\bu},\by)=u_{n+1}^{2d_{G}}G(\bu/u_{n+1}, \by)\succeq 0$, 
    we have $G(\bu/u_{n+1}, \by)\succeq 0$ for all $\by\in\Y$, which implies $\bu/u_{n+1}\in\X$.
    Hence,  $\wt{F}(\td{\bu})=u_{n+1}^{2d_{F}}F(\bu/x_{n+1})\succeq 0$. Since 
    $\Vert\td{\bu}\Vert^2\neq 0$,
    $\wt{F}(\td{\bu})+\varepsilon\Vert\td{\bu}\Vert^{2d_F}I_p\succ 0$.
    \item Case 2: $u_{n+1}= 0$. By the definition of $d_F$, $x_{n+1}$ divides $\wt{F}(\td{\bx})$.
    Thus, $\wt{F}(\td{\bu})=0$. Since $\Vert\td{\bu}\Vert^2\neq 0$, 
    $\wt{F}(\td{\bu})+\varepsilon\Vert\td{\bu}\Vert^{2d_F}I_p \succ 0$.
\end{enumerate}

Applying Theorem \ref{th::homo} to the polynomial matrices $\td{F}+\varepsilon\Vert\td{\bx}\Vert^{2d_F}I_p$
and $\wt{G}$, we obtain $N_{\varepsilon}\in\N$ and SOS matrices $\wt{S}_0\in\bS[\bx]^p$, $\wt{S}\in\bS[\bx,\by]^{pq}$, which 
are homogeneous in $\bx$ and $\deg\wt{S}_0=2N_{\varepsilon}+2d_F$, 
$\deg_{\bx}\wt{S}=2N_{\varepsilon}+2d_F-2d_G$, such that 
\begin{equation}\label{eq::homo}
\Vert\td{\bx}\Vert^{2N_{\varepsilon}}(\wt{F}(\td{\bx})+\varepsilon\Vert\td{\bx}\Vert^{2d_F}I_p)=\wt{S}_0(\td{\bx})+\int_{\Y}\lram{\wt{S}(\td{\bx},\by)}{\wt{G}(\td{\bx}, \by)}{p} \ud\nu(\by).
 \end{equation}
Then, letting $x_{n+1}=1$ in \eqref{eq::homo}, we achieve the desired conclusion.
\end{proof}

\subsection{A Positivstellensatz on the nonnegative orthant}

P{\'o}lya \cite{polya} proved that multiplying a positive homogeneous polynomial on the nonnegative 
orthant by some power of $\sum_{i=1}^n x_i$ yields a polynomial with nonnegative coefficients.
Powers and Reznick \cite{PR2001} gave a ''positive'' version of P\'olya's theorem.
Dickinson and Povh \cite{DP2015} extended the result of P{\'o}lya \cite{polya} to provide a certificate for
positive homogeneous polynomials on the intersection of the nonnegative orthant with a basic semialgebraic set.

We next employ Theorem \ref{th::homo} and similar techniques from \cite{DP2015} to 
present a Positivstellensatz for homogeneous polynomial matrices 
being PD on the intersection of $\X$ with the nonnegative orthant.
Denote by $\be$ the column vector in $\RR^n$ of all ones.
%We now prove the following P{\'o}lya-type Positivstellensatze with universal quantifiers. 
\begin{theorem}\label{th::PolyaMPU}
Suppose that 
$F\in\bS[\bx]^p$ and $G\in\bS[\bx, \by]^{q}$ are homogeneous in $\bx$,
$F(\bx)\succ 0$ for all $\bx\in\RR^n_+\cap\X\setminus\{\mathbf{0}\}$, and $\Y$ is compact. 
Then, there exists $N\in\N$ and 
polynomial matrices 
\[
P_0=\sum_{|\ba|=N+\deg F}P_{0,\ba}\bx^{\ba}\in\bS[\bx]^p,\quad 
P=\sum_{|\ba|=N+\deg F-\deg_{\bx}G}P_{\ba}(\by)\bx^{\ba}\in\bS[\bx,\by]^{pq}
\]
which 
are homogeneous in $\bx$ and satisfy $P_{0,\ba}\succ 0$ for all $|\ba|=N+\deg F$, 
$P_{\ba}(\by)\succ 0$ for all $|\ba|=N+\deg F-\deg_{\bx}G$ and $\by\in\Y$, such that 
\[
(\be^\intercal\bx)^{N}F(\bx)=P_0(\bx)+\int_{\Y}\lram{P(\bx, \by)}{G(\bx, \by)}{p}
 \ud\nu(\by).
 \]
\end{theorem}
The proof is postponed to Section \ref{sec::prfThPolya}.

In particular, when $\Y$ is a semialgebriac set defined by PMIs, 
the coefficient matrices $P_{\ba}(\by)$ in Theorem~\ref{th::PolyaMPU}
have SOS-structured representations.

\begin{corollary}\label{cor::PPMPU}
Suppose that 
%Assumption \ref{assump::carleman} holds,
$F\in\bS[\bx]^p$ and $G\in\bS[\bx, \by]^{q}$ are homogeneous in $\bx$, and
$F(\bx)\succ 0$ for all $\bx\in\RR^n_+\cap\X\setminus\{\mathbf{0}\}$. Moreover, suppose that 
\begin{equation}\label{def::Y}
\Y=\{\by\in\RR^m \mid H_1(\by)\succeq 0, \ldots, H_t(\by)\succeq 0\},\ 
H_i(\by)\in\bS[\by]^{\ell_i},\ \ell_i\in\N, \ i\in[t],
\end{equation}
and the Archimedean condition holds for the quadratic module $\QM^{pq}(\bH)$ where
$\bH\coloneqq\{H_1,\ldots,H_t\}$. Then, there exists $N\in\N$ and 
polynomial matrices 
\[
P_0=\sum_{|\ba|=N+\deg F}P_{0,\ba}\bx^{\ba}\in\bS[\bx]^p,\quad 
P=\sum_{|\ba|=N+\deg F-\deg_{\bx}G}P_{\ba}(\by)\bx^{\ba}\in\bS[\bx,\by]^{pq}
\]
which 
are homogeneous in $\bx$ and satisfy $P_{0,\ba}\succ 0$ for all $|\ba|=N+\deg F$, 
$P_{\ba}(\by)\in\QM^{pq}(\bH)$ for all $|\ba|=N+\deg F-\deg_{\bx}G$,
%and $\by\in\Y$, 
such that 
\[
(\be^\intercal\bx)^{N}F(\bx)=P_0(\bx)+\int_{\Y}\lram{P(\bx, \by)}{G(\bx, \by)}{p}
 \ud\nu(\by).
 \]
\end{corollary}
\begin{proof}
    It follows from Theorem \ref{th::PolyaMPU} and Scherer-Hol's Positivestellensatz (Theorem \ref{th::psatz}).
\end{proof}
In analogy with Corollary \ref{cor::inhomo}, we can obtain the following inhomogeneous counterpart of 
Theorem \ref{th::PolyaMPU}. We omit the similar proof for the sake of clarity. 
\begin{corollary}\label{cor::PolyaMPU}
Suppose that 
$F\in\bS[\bx]^p$ and $G\in\bS[\bx, \by]^{q}$,
$F(\bx)$ is PSD on $\RR^n_+\cap\X$, and $\Y$ is compact. 
Then for any $\varepsilon>0$, there exists $N_{\varepsilon}\in\N$ and 
polynomial matrices 
\[
P_0=\sum_{|\ba|\le N_{\varepsilon}+2d_F}P_{0,\ba}\bx^{\ba}\in\bS[\bx]^p,\quad 
P=\sum_{|\ba|\le N_{\varepsilon}+2d_F-2d_G}P_{\ba}(\by)\bx^{\ba}\in\bS[\bx,\by]^{pq}
\]
satisfying 
$P_{0,\ba}\succ 0$ for all $|\ba|\le N_{\varepsilon}+2d_F$, 
$P_{\ba}(\by)\succ 0$ for all $|\ba|\le N_{\varepsilon}+2d_F-2d_G$ and $\by\in\Y$, and 
\[
(1+\be^\intercal\bx)^{N_{\varepsilon}}\left(F(\bx)+\varepsilon(1+\be^\intercal\bx)^{2d_F}I_p\right)
=P_0(\bx)+\int_{\Y}\lram{P(\bx, \by)}{G(\bx, \by)}{p}
 \ud\nu(\by).
 \]
\end{corollary}

\section{Applications to robust PMI constrained optimization}\label{sec::app}
Verifying PMIs over a prescribed set has a wide range of applications in many fields.
For example, many control problems for systems of ordinary differential equations can be formulated as convex optimization problems with PMI constraints that must be satisfied over a specified portion of the state space \cite{HL2006,HL2012,ichihara2009optimal,pozdyayev2014atomic,vanantwerp2000tutorial}.
These problems are typically formulated as follows:
\[
\inf_{\bg\in\RR^r}\ \bc^\intercal \bg\ \text{ s.t. } \ \Lambda(\bx, \bg)\coloneqq \Lambda_0(\bx)-\sum_{i=1}^r \Lambda_i(\bx)\gamma_i\succeq 0,\ \forall \bx\in\{\bx\in\RR^n \mid H(\bx)\succeq 0\},
\]
where  $\bc\in\RR^r$, $\Lambda_0,\ldots, \Lambda_r\in\bS[\bx]^p$ and $H\in\bS[\bx]^q$. 

However, due to estimation errors or lack of information, 
data from real-world problems often involve uncertainty. 
Therefore, ensuring the robustness of PMIs over a prescribed set under uncertainty is a critical issue, which could be formulated as the following robust optimization problem:
\begin{equation}\label{eq::pmo}
\tau^{\star}\coloneqq\inf_{\bg\in\RR^r}\ \bc^\intercal\bg\ \text{ s.t. } \ \Lambda(\bx, \bg)\succeq 0,\ \forall \bx\in\X\coloneqq 
\{\bx\in\RR^n\mid G(\bx,\by)\succeq 0,\ \forall \by\in\Y\},
\end{equation}
where 
%\[
%\X\coloneqq 
%\{\bx\in\RR^n\mid G(\bx,\by)\succeq 0,\ \forall \by\in\Y\},
%\]
$G(\bx, \by)\in\bS[\bx, \by]^q$ and $\Y\subset\RR^m$ is closed.

%It is clear that the problem \eqref{eq::pmo} is equivalent to 
%\[
%\sup_{\lambda\in\RR}\ \lambda\ \text{ s.t. } F(\bx)-\lambda\succeq 0 \text{ on } \X.
%\]
Consequently, Positivstellens\"atze for polynomial matrices with UQs
developed in this paper serve as a powerful mathematical tool for addressing this issue. 
Indeed, by leveraging the SOS-structured certificates provided by Positivstellens\"atze for the positive definiteness of $\Lambda(\bx, \bg)$ over $\X$, we are able to establish hierarchies of SDP relaxations for the problem \eqref{eq::pmo}.
%and establish their convergence under the 
%conditions specified in these Positivstellens\"atze.

\subsection{The compact case} 
If $\X$ is compact, then one could apply Theorem \ref{th::SHpsatz} to construct a hierarchy of SDP relaxations for \eqref{eq::pmo}:
\begin{equation}\label{eq::sdp4pmo}
\left\{
\begin{aligned}
    \tau(k)\coloneqq \inf_{\bg,S_0,S}&\ \bc^\intercal\bg\\ 
\text{s.t.}& \ \Lambda(\bx, \bg)=S_0(\bx)+\int_{\Y}\lram{S(\bx,\by)}{G(\bx,\by)}{p}\ud\nu(\by),\\
&\ S_0\in\bS[\bx]^{p}, S\in\bS[\bx, \by]^{pq} \,\,\text{are SOS matrices},\\
&\ \deg S_0\le 2k,\ \deg_{\bx}S\le 2(k-\lceil\deg_{\bx}G/2\rceil),\ \deg_{\by}S\le 2(k-\lceil\deg_{\by}G/2\rceil).
\end{aligned}\right.
\end{equation}
It is clear that $\{\tau(k)\}_{k\in\N}$ is a sequence of non-increasing upper bounds on $\tau^{\star}$.
\begin{theorem}\label{th::sdp4pmo}
Suppose that Assumptions 
\ref{assump2} and \ref{assump::carleman} hold and there exists a point $\bar{\bg}\in\RR^r$
such that $\Lambda(\bx, \bar{\bg})\succ 0$ for all $\bx\in\X$. 
Then we have $\lim_{k\to\infty}\tau(k)=\tau^{\star}$. 
\end{theorem}
\begin{proof}
Note that for any feasible point $\bg'$ of \eqref{eq::pmo} and any $\varepsilon>0$, the point $(1-\varepsilon)\bg'+\varepsilon\bar{\bg}$ satisfies $\Lambda(\bx, (1-\varepsilon)\bg'+\varepsilon\bar{\bg})\succ 0$ on $\X$. Then, the conclusion follows by Theorem \ref{th::SHpsatz}. 
\end{proof}

Given that the SDP problem 
derived from \eqref{eq::sdp4pmo} scales rapidly and becomes computationally prohibitive as $k$ increases, 
we revise it by imposing separate degree constraints on $\bx$ and $\by$, leading to the following problem
\begin{equation}\label{eq::sdp4pmo2}
\left\{
\begin{aligned}
    \tau(k_{\bx},k_{\by})\coloneqq \inf_{\bg,S_0,S}&\ \bc^\intercal\bg\\ 
\text{s.t.}& \ \Lambda(\bx, \bg)=S_0(\bx)+\int_{\Y}\lram{S(\bx,\by)}{G(\bx,\by)}{p}\ud\nu(\by),\\
&\ S_0\in\bS[\bx]^{p}, S\in\bS[\bx, \by]^{pq} \,\,\text{are SOS matrices},\\
&\ \deg S_0\le 2k_{\bx},\ \deg_{\bx}S\le 2(k_{\bx}-\lceil\deg_{\bx}G/2\rceil),\ \deg_{\by}S\le 2(k_{\by}-\lceil\deg_{\by}G/2\rceil).
\end{aligned}\right.
\end{equation}
If $\Y$ is low-dimensional, we can maintain a small $k_{\bx}$ while increasing $k_{\by}$ separately; 
this allows us to solve \eqref{eq::sdp4pmo2} in reasonable time and obtain a nontrivial upper bound 
$\tau(k_{\bx},k_{\by})$ of $\tau^{\star}$.

In practice, the SDP relaxation \eqref{eq::sdp4pmo2} requires an efficient evaluation of the integral 
$\int_{\Y}\by^{\bb}\ud\nu(\by)$. If $\nu$ is chosen to be the Lebesgue measure on $\Y$ (as in the 
following illustrating examples), this integral is available in closed form or by numerical methods 
whenever $\Y$ is a simple domain; examples include hypercubes, spheres \cite{Folland2001}, 
balls \cite{Folland2001}, and polytopes \cite{LattE}. 
All numerical experiments in the sequel were carried out on a PC with  
an 8-Core Intel i7 3.8GHz CPU and 8G RAM, using the SDP solver MOSEK \cite{mosek} and
a customized version of the MATLAB optimization toolbox YALMIP \cite{YALMIP}. 
%All scripts and data used to generate the results presented below are available from \url{https:}.

\begin{example}\label{ex::1}{\rm
%Let $\Lambda(\bx,\bg)$ be a $p\times p$ polynomial matrix which is linear in $\bg$ 
%and $\X$ be defined by PMIs under uncertainty, 
%\[
%\X=\{\bx\in\RR^n\mid G(\bx,\by)\succeq 0,\ \forall \by\in\Y\}, \ \ \Y\subset\RR^m.
%\]
Consider the feasible set of the problem \eqref{eq::pmo}:
\[
\mathcal{F}=\{\bg\in\RR^r\mid \Lambda(\bx,\bg)\succeq 0, \forall \bx\in\X\}.
\]
In this example, we illustrate how Theorem \ref{th::SHpsatz} enables one to approximate the 
convex set of vectors $\bg$ for which $\Lambda(\bx,\bg)$ is PSD on $\X$ -- that is, 
to obtain approximate semidefinite representations of $\mathcal{F}$.

For any $k_{\bx}, k_{\by}\in\N$ with $k_{\bx}\ge\lceil\deg_{\bx}G/2\rceil$ and 
$k_{\by}\ge\lceil\deg_{\by}G/2\rceil$, let $\nu$ in \eqref{eq::sdp4pmo2} be the Lebesgue measure on $\Y$ and consider 
\[
\mathcal{F}(k_{\bx}, k_{\by})=
\left\{
\begin{array}{c|c}
   \bg\in\RR^r  &
   \begin{aligned}
&\Lambda(\bx, \bg)=S_0(\bx)+\int_{\Y}\lram{S(\bx,\by)}{G(\bx,\by)}{p}\ud\by,\\
&S_0\in\bS[\bx]^{p}, S\in\bS[\bx, \by]^{pq} \,\,\text{are SOS matrices},\\
&\deg S_0\le 2k_{\bx},\ \deg_{\bx}S\le 2(k_{\bx}-\lceil\deg_{\bx}G/2\rceil),\ \deg_{\by}S\le 2(k_{\by}-\lceil\deg_{\by}G/2\rceil).
\end{aligned}
\end{array}
\right\}.
\]
Clearly, $\mathcal{F}(k_{\bx}, k_{\by})$ is an inner approximation of $\mathcal{F}$. 
On the other hand, Theorem \ref{th::SHpsatz} guarantees that 
any $\bg\in\RR^r$ for which $\Lambda(\bx,\bg)$ is PD on $\X$ belongs to $\mathcal{F}(k_{\bx}, k_{\by})$
for sufficiently large $k_{\bx}$ and $k_{\by}$. Therefore, the sets $\mathcal{F}(k_{\bx}, k_{\by})$ 
can approximate $\mathcal{F}$ from inside arbitrarily accurately in the sense that any compact subset 
of the interior of $\mathcal{F}$ is included in $\mathcal{F}(k_{\bx}, k_{\by})$ for some sufficiently 
large integers $k_{\bx}$ and $k_{\by}$.

Let us construct an illustrating example. Let $n=r=2$, $m=1$, $\Y=[0, 1]$,
    \[
    G(\bx,y)=
    \left[
    \begin{array}{cc}
        (1/2-y)^2 & y(x_1^2+x_2^2-1) \\
        y(x_1^2+x_2^2-1) & x_1+x_2-1
    \end{array}
    \right],
    \]
    and
    \[
    \Lambda(\bx,\bg)=
    \left[
    \begin{array}{cc}
        1-(x_1\gamma_1-x_2\gamma_2+1/2) & 2(x_1\gamma_2+x_2\gamma_1) \\
        2(x_1\gamma_2+x_2\gamma_1) & 1+(x_1\gamma_1-x_2\gamma_2+1/2) 
    \end{array}
    \right].
    \]
It is clear that 
\[
\X=\{\bx\in\RR^2 \mid x_1^2+x_2^2=1,\ x_1x_2\ge 0\}.
\]
Geometrically, the set $\mathcal{F}$ is constructed as the intersection of all rotated copies of 
the shape defined by $(\gamma_1+1/2)^2+4\gamma_2^2\le 1$ in the $\bg$-plane as it is continuously 
rotated clockwise by $90^\circ$ about the origin. The resulting set $\mathcal{F}$ (light blue region 
with blue outline) is shown in Figure \ref{fig::ex1}. 
Three inner approximations -- $\mathcal{F}(1,1)$, $\mathcal{F}(1,2)$, and $\mathcal{F}(1,6)$
-- each depicted with a red outline, appear in Figures \ref{fig::ex1}. 
These red outlines were obtained by solving and extracting optimal solutions of $\eqref{eq::pmo}$ with
objective functions $\gamma_1\cos\theta + \gamma_2\sin\theta$ over $80$ equispaced values of $\theta$ 
in the interval $[0,2\pi]$.
\begin{figure}
\centering
\scalebox{0.5}{
\includegraphics[trim=60 195 80 190,clip]{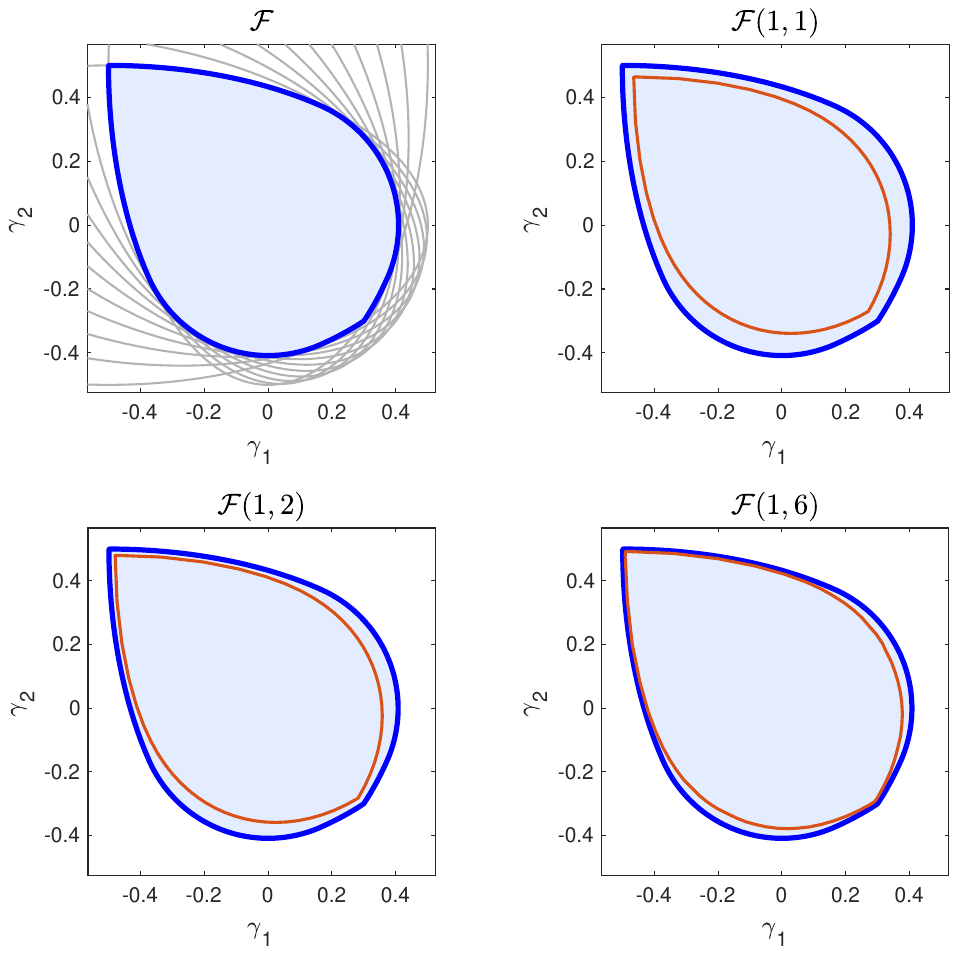}}
\caption{The region $\mathcal{F}$ and three inner approximations $\mathcal{F}(1,1)$, $\mathcal{F}(1,2)$,
$\mathcal{F}(1,6)$ in Example~\ref{ex::1}.}\label{fig::ex1}
\end{figure} 
  \qed  }
\end{example}

%\begin{example}{\rm
%Consider the problem of minimizing the smallest eigenvalue of a polynomial matrix 
%$F(\bx)\in\bS[\bx]^p$ over the set $\X$. Clearly, this problem is equivalent to 
%\begin{equation}\label{eq::empo2}
%\sup_{\gamma\in\RR}\,\gamma\quad \text{s.t.}\quad F(\bx)-\gamma I_p\succeq 0,\quad \forall\bx\in\X,
%\end{equation}
%which is a special case of \eqref{eq::pmo}. 
%
%    Let $n=2$, $m=r=1$, $\Y=[0, 1]$
%    \[
%    G(\bx, y)=
%    \left[
%    \begin{array}{cc}
%     1-\frac{\sqrt{2}}{2}y(x_1+x_2)  & \sqrt{2}(x_1-x_2) \\
%     \sqrt{2}(x_1-x_2)    & 1+\frac{\sqrt{2}}{2}y(x_1+x_2)
%    \end{array}\right],
%    \quad
%    F(\bx)=
%    \left[
%    \begin{array}{cc}
%     x_1^2+x_2^2-2x_2+2  & -2x_1 \\
%     -2x_1    & x_1^2+x_2^2-2x_2+2
%    \end{array}\right],
%    \]
%    \qed}
%\end{example}
\subsection{Exploiting sparsity}
We illustrate the computational advantage of our sparsity-exploiting Positivstellensatz 
when the correlative sparsity pattern in Assumption \ref{assump_CS} is satisfied.
Consider the problem 
\[
f^{\star}\coloneqq\min\ f(\bx)\ \text{ s.t. } \bx\in\widehat{\X},
\]
where $f(\bx)\in\RR[\bx]$ and the set $\widehat{\X}$ is defined in \eqref{def::sparseX}. 
By Theorem \ref{th::SHpsatz}, dense SDP relaxations for such a problem can be formulated as 
\begin{equation}\label{eq::minf}
\left\{
\begin{aligned}
    \tau(k_{\bx}, k_{\by})\coloneqq \sup_{\gamma,\sigma_0,S_j}&\ \gamma\\ 
\text{s.t.}& \ f(\bx)-\gamma=\sigma_0(\bx)+\sum_{j=1}^s \int_{\Y}
\left\langle S_j(\bx,\by), G_j(\bx,\by)\right\rangle\ud\nu(\by),\\
&\ \sigma_0\in\RR[\bx], S_j\in\bS[\bx, \by]^{q_j},\, j\in[s],\,\ \ \text{are SOS},\\
&\ \deg \sigma_0\le 2k_{\bx},\ \deg_{\bx}S_j\le 2(k_{\bx}-\lceil\deg_{\bx}G_j/2\rceil),\ \deg_{\by}S_j\le 2(k_{\by}-\lceil\deg_{\by}G_j/2\rceil).
\end{aligned}\right.
\end{equation}

For any $k_{\bx}\ge \lceil\deg_{\bx}G_j/2\rceil$ and 
$k_{\by}\ge \lceil\deg_{\by}G_j/2\rceil$, the number $\tau(k_{\bx}, k_{\by})$ provides a lower bound of $f^{\star}$.
The sequence $\{\tau(k,k)\}_{k\in\N}$, the limit $\tau(\infty,\infty)\coloneqq\lim_{k\to\infty}\tau(k,k)$ exists. By Theorem \ref{th::sdp4pmo}, we have $\tau(\infty,\infty)=\tau^{\star}$ if Assumptions \ref{assump2} and \ref{assump::carleman} hold.

Provided the presence of the correlative sparsity pattern stated in 
Assumption \ref{assump_CS}, sparse SDP relaxations
for such a problem can be constructed as follows:
\begin{equation}\label{eq::minfcs}
\left\{
\begin{aligned}
    \tau^{\csp}(k_{\bx}, k_{\by})\coloneqq \sup_{\gamma,\sigma_0,S_j}&\ \gamma\\ 
\text{s.t.}& \ f(\bx)-\gamma=\sum_{\ell=1}^t\left( \sigma_{\ell,0}(\bx(\mI_{\ell}))+ \sum_{j\in\mJ_{\ell}}\int_{\Y}
\left\langle S_{j}(\bx(\mI_{\ell}),\by), G_j(\bx(\mI_{\ell}),\by)\right\rangle\ud\nu(\by)\right),\\
&\ \sigma_{\ell,0}\in\RR[\bx(\mI_{\ell})], S_{j}\in\bS[\bx(\mI_{\ell}), \by]^{q_j},
\ j\in\mJ_{\ell},\ \ell\in[t],\, \text{are SOS},\\
&\ \deg \sigma_{\ell, 0}\le 2k_{\bx},\ \deg_{\bx}S_j\le 2(k_{\bx}-\lceil\deg_{\bx}G_j/2\rceil),\ \deg_{\by}S_j\le 2(k_{\by}-\lceil\deg_{\by}G_j/2\rceil).
\end{aligned}\right.
\end{equation}
For the sequence $\{\tau^{\csp}(k,k)\}_{k\in\N}$, if the assumptions in Theorem \ref{th::sparse} hold, we have the convergence 
\[\tau^{\csp}(\infty,\infty)\coloneqq\lim_{k\to\infty}\tau^{\csp}(k,k)=\tau^{\star}.
\]
When the number $\max\{|\mI_{\ell}|, \ell\in[t]\}$ is small, 
dramatic computational savings can be expected.
\begin{example}\label{ex::cs}{\rm
Consider the problem $f^{\star}\coloneqq\min_{\bx\in\widehat{X}}f(\bx)$ where $f(\bx)=\sum_{i=1}^n(x_i+1)^4$,
    \[
    \widehat{\X}=\{\bx\in\RR^n\mid G_j(\bx, y)\succeq 0, \ \forall j\in[n-1],\ y\in\Y\},\quad
    G_j(\bx, y)=
    \left[
    \begin{array}{cc}
     1-yx_j  & x_{j+1} \\
     x_{j+1}    & 1+yx_j
    \end{array}\right],\ j\in[n-1],
    \]
and $\Y=[0, 1]$. 
Clearly, Assumption \ref{assump_CS} is satisfied with $t=n-1$, $\mI_{\ell}=\{\ell,\ell+1\}$ and 
$\mJ_{\ell}=\{\ell\}$ for $\ell\in[n-1]$.
It is easy to see that 
\[
\widehat{\X}=\{\bx\in\RR^n\mid x_i^2+x_{i+1}^2\le 1,\ i\in[n-1]\}.
\]
Since the problem $\min_{\bx\in\widehat{X}}f(\bx)$ is convex, numerical values for $f^{\star}$ 
can be obtained by solving the associated KKT system. These results can then be used to evaluate 
the accuracy of the SDP lower bounds $\tau(k_{\bx}, k_{\by})$ and $\tau^{\csp}(k_{\bx}, k_{\by})$.
We choose $\nu$ in the SDP relaxations \eqref{eq::minf} and \eqref{eq::minfcs} to be the 
Lebesgue measure on $\Y$ and present the numerical results in Table \ref{tab::1}, where 
the symbol `--/--' indicates that MOSEK runs out of memory. From the table, we could see that 
dramatic computational advantages are gained by exploiting correlative sparsity.
\qed

\begin{table}[htb]\caption{Computational results for Example \ref{ex::cs}.}{\small
	\centering
	\begin{tabular}{ccccccccc}
\midrule[0.8pt]
\multirow{2}{*}{$n$} && \multirow{2}{*}{$f^{\star}$} &
&\multicolumn{2}{c}{$(k_{\bx},k_{\by}$)=(2, 10)} & &\multicolumn{2}{c}{$(k_{\bx},k_{\by}$)=(2, 15)}  \\
		\cline{5-6} \cline{8-9} 
&& & &
$\tau(k_{\bx}, k_{\by})$/time & $\tau^{\csp}(k_{\bx}, k_{\by})$/time
& &
$\tau(k_{\bx}, k_{\by})$/time & $\tau^{\csp}(k_{\bx}, k_{\by})$/time\\
\midrule[0.4pt]
		$3$&& 0.0206 &&0.0184/13.1s& 0.0184/4.45s && 0.0191/83.3s&0.0191/19.9s\\
        $4$&& 0.0294 &&0.0262/65.3s& 0.0262/9.23s && --/-- &0.0273/36.5s\\
        $5$&& 0.0359 &&--/--& 0.0320/13.9s && --/--&0.0333/45.1s\\
        $6$&& 0.0441 &&--/--& 0.0394/19.2s && --/--&0.0409/62.5s\\
\midrule[0.8pt]
	\end{tabular}
		\label{tab::1}
	}
\end{table}
}
\end{example}

%Other Positivstellens\"atze presented in this paper can also be employed to derive the corresponding hierarchies of SDP relaxations for \eqref{eq::pmo}. The study of convergence of these hierarchies provides an intriguing avenue for future research. For instance, 
\subsection{The non-compact case}
Let us consider the problem \eqref{eq::pmo} with a non‑compact set $\X$. For such a problem, 
after fixing a constant $\varepsilon > 0$, we can apply Corollary~\ref{cor::inhomo} to 
construct the following hierarchy of SDP relaxations:
\begin{equation}\label{eq::sdp4nc}
\left\{
\begin{aligned}
    \tau_{\varepsilon}(k_{\bx},k_{\by})\coloneqq \inf_{\bg,S_0,S}&\ \bc^\intercal\bg\\ 
\text{s.t.}& \ \theta^{k_{\bx}-d_{\Lambda}} (\Lambda(\bx, \bg)+\varepsilon\theta^{d_{\Lambda}}I_p)=S_0(\bx)+\int_{\Y}\lram{S(\bx,\by)}{G(\bx,\by)}{p}\ud\nu(\by),\\
&\ S_0\in\bS[\bx]^{p}, S\in\bS[\bx, \by]^{pq} \,\,\text{are SOS matrices},\\
&\ \deg S_0\le 2k_{\bx},\ \deg_{\bx}S\le 2(k_{\bx}-d_G),\ \deg_{\by}S\le 2(k_{\by}-\lceil \deg_{\by}G/2\rceil),
\end{aligned}\right.
\end{equation}
where $d_{\Lambda}\coloneqq\max_{i,j\in[p]}\lfloor\deg_{\bx}(\Lambda_{ij})/2\rfloor+1$. 
Let $\mathcal{F}_{\varepsilon}(k_{\bx},k_{\by})$ denote the feasible set of \eqref{eq::sdp4nc}.

For the sequence $\{\tau_{\varepsilon}(\ell,\ell)\}_{\varepsilon>0,\ell\in\N}$, let $\tau_0(\infty,\infty)\coloneqq\lim_{\varepsilon\to 0}\lim_{\ell\to\infty}\tau_{\varepsilon}(\ell,\ell)$. 
\begin{theorem}\label{th::conv_nc}
    Let Assumption~\ref{assump::carleman} hold. Then the limit $\tau_0(\infty,\infty)$ exists, and 
    $\tau_0(\infty,\infty)\le \tau^{\star}$. 
    Moreover, if the optimal solution set $\mathcal{S}$ of \eqref{eq::pmo} is nonempty and bounded, 
    we have $\tau_0(\infty,\infty)=\tau^{\star}$.
\end{theorem}
\begin{proof}
    Since Assumption~\ref{assump::carleman} holds, Corollary~\ref{cor::inhomo} implies that, 
for any $\varepsilon>0$ and sufficiently large $\ell$,
$\mathcal{F}_{\varepsilon}(\ell,\ell)$ contains the feasible set $\mathcal{F}$ 
of \eqref{eq::pmo} and hence $\tau_{\varepsilon}(\ell,\ell) \le \tau^{\star}$. 
It is clear that $\tau_{\varepsilon}(\ell,\ell)$ is non-increasing as $\ell\to\infty$
for fixed $\varepsilon>0$ and is non-decreasing as $\varepsilon\to 0$ for fixed $\ell$.
Hence, the limit $\tau_0(\infty,\infty)$ exists and $\tau_0(\infty,\infty)\in[-\infty, \tau^{\star}]$.

Assuming that the optimal solution set $\mathcal{S}$ of \eqref{eq::pmo} is nonempty and bounded, we 
show that $\tau_0(\infty,\infty)=\tau^{\star}$. Suppose on the contrary that 
$\tau_0(\infty,\infty)<\tau^{\star}$. Then, there exists a sequence $\{\varepsilon_j\}_{j\in\N}$
with $\lim_{j\to\infty}\varepsilon_j=0$ and a sequence $\{\bg^{(\ell_j)}\}_{j\in\N}$ of feasible points 
of \eqref{eq::sdp4nc} with $(\varepsilon,k_{\bx},k_{\by})=(\varepsilon_j, \ell_j, \ell_j)$ such that 
for each $j\in\N$, $\bc^\intercal \bg^{(\ell_j)}\le \tau^{\star}-\delta$ for some $\delta>0$.
We claim that $\{\bg^{(\ell_j)}\}_{j\in\N}$ is bounded. Suppose that 
$\Vert\bg^{(\ell_j)}\Vert\to\infty$ as $j\to\infty$ and without loss of generality that 
$\lim_{j\to\infty}\bg^{(\ell_j)}/\Vert\bg^{(\ell_j)}\Vert=\bv$ for some $\bv=(v_1,\ldots,v_r)\in\RR^r$ 
with $\Vert \bv\Vert=1$. Since $\bc^\intercal \bg^{(\ell_j)}$ is bounded by $\tau^{\star}-\delta$ 
for all $j\in\N$,
we have $\bc^\intercal \bv=0$. Moreover, for arbitrarily fixed $\bu\in\X$, we have 
\begin{equation}\label{eq::feasible}
\Lambda(\bu, \bg^{(\ell_j)})+\varepsilon_j\theta(\bu)^{d_{\Lambda}}I_p=
\Lambda_0(\bu)-\sum_{i=1}^r \Lambda_i(\bu)\gamma^{(\ell_j)}_i+\varepsilon_j\theta(\bu)^{d_\Lambda}I_p
\succeq 0, \ \forall j\in\N.
\end{equation}
Dividing the left-hand side of \eqref{eq::feasible} by $\Vert\bg^{(\ell_j)}\Vert$ and letting $j\to\infty$,
we get $-\sum_{i=1}^r \Lambda_i(\bu)v_i\succeq 0$. Since $\bu$ is an arbitrary point in $\X$ and $\bc^\intercal \bv=0$, 
for any point $\bg^{\star}\in\mathcal{S}$, we have $\bg^{\star}+t\bv\in\mathcal{S}$ for all $t>0$,
which contradicts the boundedness of $\mathcal{S}$. As we have proved that 
$\{\bg^{(\ell_j)}\}_{j\in\N}$ is bounded, we can assume without loss of generality that 
$\lim_{j\to\infty}\bg^{(\ell_j)}=\hat{\bg}$ for some $\hat{\bg}\in\RR^r$. Then letting 
$j\to\infty$ in \eqref{eq::feasible} yields the result that $\hat{\bg}$ is feasible to \eqref{eq::pmo}.
Thus, we have
\[
\tau^{\star}\le \bc^\intercal \hat{\bg}=\lim_{j\to\infty}\bc^\intercal \bg^{(\ell_j)}\le \tau^{\star}-\delta,
\]
a contradiction.
\end{proof}

%How to bound the number $\bar{k}$ from above, and whether $\tau_k(\varepsilon)\to\tau^{\star}$ as $\varepsilon\to 0$ and $k\to\infty$ under certain conditions, could be interesting topics for further study.
Now, let us consider the problem 
\begin{equation}\label{eq::pmo+}
\td{\tau}^{\star}\coloneqq\inf_{\bg\in\RR^r}\ \bc^\intercal\bg\ \text{ s.t. } \ \Lambda(\bx, \bg)\succeq 0,\ \forall \bx\in\X\cap\RR_+^n,
\end{equation}
where $\X$ is non-compact and defined in \eqref{eq::pmo}, and $\Y$ is defined
in \eqref{def::Y}.
After fixing a constant $\varepsilon > 0$, we apply Corollary~\ref{cor::PolyaMPU} to 
construct the following hierarchy of SDP relaxations:
\begin{equation}\label{eq::sdp4no}
\left\{
\begin{aligned}
    \tilde{\tau}_{\varepsilon}(k_{\bx},k_{\by})\coloneqq \inf_{\bg,S_0,S}&\ \bc^\intercal\bg\\ 
\text{s.t.}& \ (1+\be^\intercal\bx)^{2k_{\bx}-2d_{\Lambda}} (\Lambda(\bx, \bg)+\varepsilon(1+\be^\intercal\bx)^{2d_{\Lambda}}I_p)
=P_0(\bx)+\int_{\Y}\lram{P(\bx,\by)}{G(\bx,\by)}{p}\ud\nu(\by),\\
&\ P_0=\sum_{\ba\in\N^n_{2k_{\bx}}}P_{0,\ba}\bx^{\ba}\in\bS[\bx]^{p} 
\text{ with each } P_{0,\ba}\succeq 0,\\
&\ P=\sum_{\ba\in\N^n_{2(k_{\bx}-d_G)}}P_{\ba}(\by)\bx^{\ba}\in\bS[\bx, \by]^{pq},\ \deg_{\by}P\le 
2(k_{\by}-\lceil \deg_{\by}G/2\rceil),\\
&\ \text{ with each } P_{\ba}(\by)\in\mathcal{Q}^{pq}_{k_{\by}-\lceil \deg_{\by}G/2\rceil}(\bH).
\end{aligned}\right.
\end{equation}

For the sequence $\{\td{\tau}_{\varepsilon}(\ell,\ell)\}_{\varepsilon>0,\ell\in\N}$, let $\td{\tau}_0(\infty,\infty)\coloneqq\lim_{\varepsilon\to 0}\lim_{\ell\to\infty}\td{\tau}_{\varepsilon}(\ell,\ell)$. 
Recall the set $\bH$ in Corollary \ref{cor::PPMPU}. Similarly to Theorem \ref{th::conv_nc}, 
we have the following convergence result.
\begin{theorem}
    Let the Archimedean condition hold for the quadratic module $\QM^{pq}(\bH)$. Then 
    the limit $\td{\tau}_0(\infty,\infty)$ exists and $\td{\tau}_0(\infty,\infty)\le \tau^{\star}$. 
    Moreover, if the optimal solution set $\mathcal{S}$ of \eqref{eq::pmo+} is nonempty and bounded, 
    we have $\td{\tau}_0(\infty,\infty)=\tau^{\star}$.
\end{theorem}
Note that we set the same degree bounds on $\bx$ and $\by$ in the representations of $\Lambda(\bx,\by)$ in 
\eqref{eq::sdp4nc} and \eqref{eq::sdp4no}, so that we can compare the behaviors of 
$\tau_{\varepsilon}(k_{\bx},k_{\by})$ and $\tilde{\tau}_{\varepsilon}(k_{\bx},k_{\by})$ when applying 
both \eqref{eq::sdp4nc} and \eqref{eq::sdp4no} to \eqref{eq::pmo} if $\X=\X\cap\RR^n_+$.

\begin{example}\label{ex::nc}{\rm
    Consider the problem \eqref{eq::pmo} where we let $m=1$, $n=r=2$, $\Y=[0, 1]$, 
\[
G(\bx, y)=
    Q_1^\intercal\left[
    \begin{array}{cc}
        x_1+x_2 & 0 \\
        0 &-2y(1-y)+(1-y)^2x_1+y^2x_2
    \end{array}
    \right]Q_1,
    \]
    and
    \[
    \Lambda(\bx, \bg)=
    Q_2^\intercal\left[
    \begin{array}{cc}
    -2+\left(-\frac{\gamma_1}{2}+\frac{3\gamma_2}{2}+3\right)x_2+\left(\frac{3\gamma_1}{2}-\frac{\gamma_2}{2}+3\right)x_1 & 0 \\
    0 & -2-\gamma_1 x_2-\gamma_2 x_1
    \end{array}
    \right]Q_2.
\]
Here, $Q_1$ and $Q_2$ are two fixed arbitrary $2\times 2$ orthogonal matrices. 

For any fixed $(x_1,x_2)$, consider the function
\[
f(y) = -2y(1-y)+(1-y)^2x_1+y^2x_2=(x_1+x_2+2)y^2 -2(x_1+1)y + x_1.
\]
The evaluations $f(0)=x_1$ and $f(1)=x_2$ imply that $x_1\ge 0$ and $x_2\ge 0$. 
Assuming $x_1,x_2\ge 0$, the minimum of $f(y)$ is 
\[
f\left(\frac{x_1+1}{x_1+x_2+2}\right) = \frac{x_1x_2-1}{x_1+x_2+2}\ \text{ where }\ \frac{x_1+1}{x_1+x_2+2}\in\Y.
\]
Thus we have 
\[
\X=\{(x_1,x_2)\in\RR^2 \mid x_1\ge 0,\ x_2\ge 0,\ x_1x_2\ge 1\}.
\]
Then similarly, we can show that
$\Lambda(\bx, \bg)\succeq 0$ for all $\bx\in\X\cap\RR_+^n$
if and only if
\[
-2+\left(-\frac{\gamma_1}{2}+\frac{3\gamma_2}{2}+3\right)x_2+\left(\frac{3\gamma_1}{2}-\frac{\gamma_2}{2}+3\right)x_1\ge 0
\ \text{ and }\ 
-2-\gamma_1 x_2-\gamma_2 x_1\ge 0
\]
on $\X$ if and only if 
\[
\frac{\gamma_1}{2}-\frac{3\gamma_2}{2}-3\le 0,\ -\frac{3\gamma_1}{2}+\frac{\gamma_2}{2}-3\le 0,\ 
\phi_1(\bg)\coloneqq\left(\frac{\gamma_1}{2}-\frac{3\gamma_2}{2}-3\right)\left(-\frac{3\gamma_1}{2}+\frac{\gamma_2}{2}-3\right)-1\ge 0,
\]
and
\[
\gamma_1\le 0,\ \gamma_2\le 0,\ \phi_2(\bg)\coloneqq\gamma_1\gamma_2-1\ge 0.
\]
The feasible set $\mathcal{F}$ is shown in light blue in Figure \ref{fig::2}.
Since $\X=\X\cap\RR^2_+$, we can apply both SDP relaxations \eqref{eq::sdp4nc} and \eqref{eq::sdp4no}
to this problem and compare the behaviors of $\tau_{\varepsilon}(k_{\bx},k_{\by})$ and 
$\td{\tau}_{\varepsilon}(k_{\bx},k_{\by})$. 

Consider the objective function $\bc^\intercal\bg=\gamma_1-\gamma_2$.
One can check that the optimal solution is one of the solutions of $\phi_1(\bg)=\phi_2(\bg)=0$:
\[
\bg^{\star}=\left(1-\sqrt{5}-\sqrt{5-2\sqrt{5}}, 1-\sqrt{5}+\sqrt{5-2\sqrt{5}}\right)\approx (-1.9626,-0.5095),
\]
and the optimal value $\tau^{\star}\approx -1.4531$. The computational results for 
$\tau_{\varepsilon}(k_{\bx}, k_{\by})$ and $\td{\tau}_{\varepsilon}(k_{\bx}, k_{\by})$ 
with different $\varepsilon$ and $(k_{\bx}, k_{\by})$ are presented in 
Table~\ref{tab::2}. As expected, for a given $\varepsilon>0$, both $\tau_{\varepsilon}(k_{\bx}, k_{\by})$ and 
$\td{\tau}_{\varepsilon}(k_{\bx}, k_{\by})$ provide lower bounds for the optimal value $\tau^{\star}$ 
when $k_{\bx}$ and $k_{\by}$ are sufficiently large 
(except for $\tau_{0.01}(2, 15)$ and $\td{\tau}_{0.01}(2, 15)$ which requires larger
$k_{\bx}$ and $k_{\by}$ as $\varepsilon=0.01$ is small). 
Moreover, for a fixed $\varepsilon>0$, 
both quantities are non-increasing as $k_{\bx}$ and $k_{\by}$ increase, and for fixed $k_{\bx}$ and $k_{\by}$, 
they are non-decreasing as $\varepsilon$ decreases. Compared with $\tau_{\varepsilon}(k_{\bx}, k_{\by})$,
the values $\td{\tau}_{\varepsilon}(k_{\bx}, k_{\by})$ with the same $(\varepsilon,k_{\bx},k_{\by})$ 
are smaller and can be computed in  much less time.

To illustrate the behavior of optimal solutions of \eqref{eq::sdp4nc} and \eqref{eq::sdp4no} under 
varying $\varepsilon$ and $(k_{\bx}, k_{\by})$, like in Example~\ref{ex::1}, we solve \eqref{eq::sdp4nc} and 
\eqref{eq::sdp4no} with the objective function $\gamma_1\cos\theta + \gamma_2\sin\theta$ for $40$ evenly spaced 
values of $\theta$ over $[0, 2\pi]$, and extract the corresponding optimal solutions. 
The feasible sets of \eqref{eq::sdp4nc} and \eqref{eq::sdp4no}, which serve as approximations of $\mathcal{F}$, 
are then outlined by connecting these optimal solutions and shown in Figure~\ref{fig::2}. As one can see,
the behavior of these optimal points aligns with the quantities $\tau_{\varepsilon}(k_{\bx}, k_{\by})$ and 
$\td{\tau}_{\varepsilon}(k_{\bx}, k_{\by})$, confirming the theoretical expectations discussed above.
In particular, it is evident from the figure that the feasible set of \eqref{eq::sdp4no} contains that of \eqref{eq::sdp4nc} with the same $(\varepsilon,k_{\bx},k_{\by})$, which corresponds to the fact that $\td{\tau}_{\varepsilon}(k_{\bx}, k_{\by})$ is 
smaller than $\tau_{\varepsilon}(k_{\bx}, k_{\by})$.
\qed

\begin{table}[htb]\caption{Computational results for Example \ref{ex::nc}.}{\small
	\centering
	\begin{tabular}{cccccccccc}
\midrule[0.8pt]
\multirow{2}{*}{$\varepsilon$} &&
\multicolumn{2}{c}{$(k_{\bx},k_{\by}$)=(2, 5)} & &\multicolumn{2}{c}{$(k_{\bx},k_{\by}$)=(2, 10)} 
& &\multicolumn{2}{c}{$(k_{\bx},k_{\by}$)=(2, 15)}\\
		\cline{3-4} \cline{6-7} \cline{9-10} 
&&
$\tau(k_{\bx}, k_{\by})$/time & $\td{\tau}(k_{\bx}, k_{\by})$/time
& &
$\tau(k_{\bx}, k_{\by})$/time & $\td{\tau}(k_{\bx}, k_{\by})$/time
& &
$\tau(k_{\bx}, k_{\by})$/time & $\td{\tau}(k_{\bx}, k_{\by})$/time\\
\midrule[0.4pt]
		$0.1$ && -1.3716/2.20s & -1.4922/1.69s && -1.4844/32.3s & -1.6587/7.32s && -1.5650/547s & -1.6837/27.2s \\
        $0.05$ && -1.2664/2.34s & -1.3163/1.93s && -1.3698/31.1s & -1.4899/7.21s && -1.4552/507s & -1.5230/29.9s \\
        $0.01$ && -1.1873/2.57s & -1.1806/2.17s && -1.2949/33.3s & -1.3625/7.86s && -1.3721/451s & -1.3900/35.7s \\
\midrule[0.8pt]
	\end{tabular}
	
		\label{tab::2}
	}
\end{table}

\begin{figure}
\centering
\scalebox{0.7}{
\includegraphics[trim=60 195 60 200,clip]{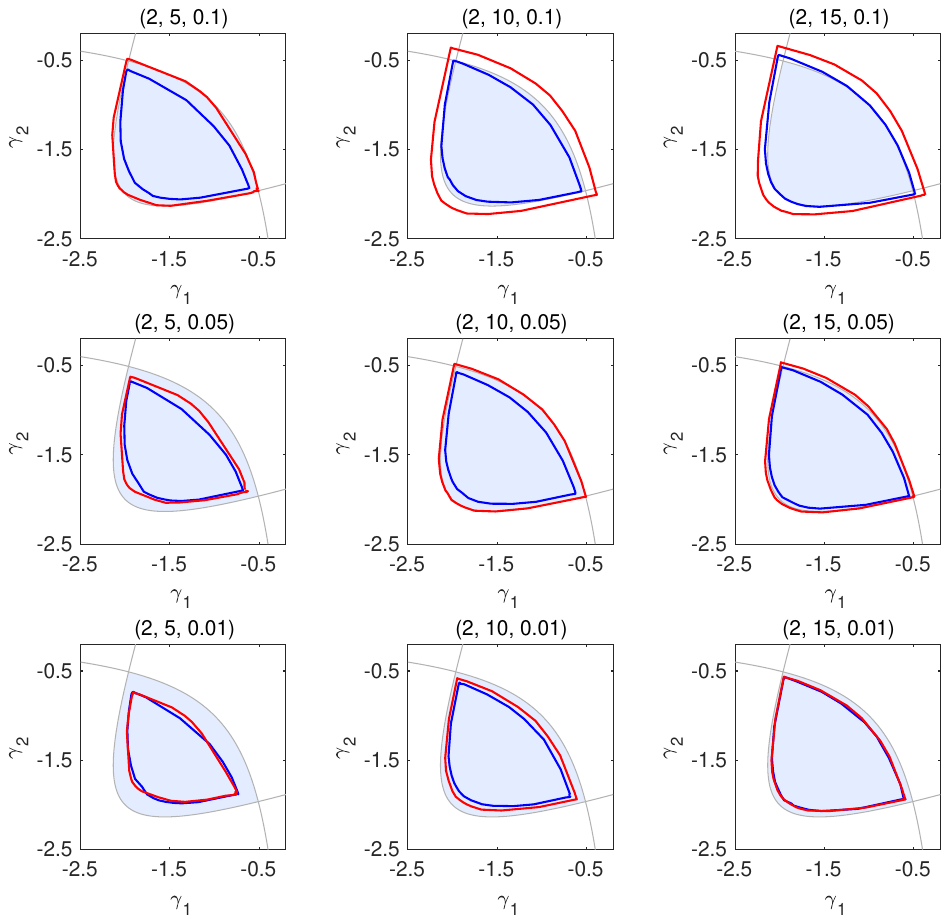}}
\caption{The feasible set (in light blue) of the problem in Example \ref{ex::nc} and the feasible sets of its SDP relaxations \eqref{eq::sdp4nc} (outlined with blue line) and \eqref{eq::sdp4no} (outlined with red line) with the same values $(\varepsilon,k_{\bx},k_{\by})$.}\label{fig::2}
\end{figure}
}
\end{example}

\section{Proofs}\label{sec::proofs}

\subsection{Proof of Proposition \ref{prop::main}}\label{sec::prfPropmain}

For a matrix $N\in\bS^q$, denote by $\lambda_{\max}(N)$ and 
$\lambda_{\min}(N)$ the largest and smallest eigenvalues of 
$N$, respectively.
\begin{proof}[Proof of Proposition \ref{prop::main}]
    For each $j\in[s]$, let 
\[ 
M\coloneqq\max\left\{|\lambda_{\max} (G_j(\bx, \by))|, |\lambda_{\min} (G_j(\bx, \by))| \colon \bx\in
   [-C, C]^n,\ \by\in\Y,\ j\in[s]\right\}. 
   \] 
   As $\Y$ is compact, the quantity $M$ is well-defined and positive. 
   For this $M$, we show that there exists some $\bar{k}\in\N$ such that \eqref{eq::ge} holds for all $k\ge\bar{k}$. Suppose on
   the contrary that for any $k\in\N$, there exists $\bx^{(k)}\in[-C,
   C]^n$ such that 
\begin{equation}\label{eq::np}
   F(\bx^{(k)})-\sum_{j=1}^s\int_{\Y}\lram{I_p\otimes \left(I_{q_j}-G_j(\bx^{(k)},
   \by)/M\right)^{2k}}{G_j(\bx^{(k)}, \by)}{p}\ud\nu(\by)\not\succ 0.
    \end{equation}
   As $[-C, C]^n$ is compact, without loss of generality, we may assume that
   $\lim_{k\to\infty}\bx^{(k)}=\bx^{\star}$ for some $\bx^{\star}\in[-C, C]^n$.
%   By continuity, \eqref{eq::np} still holds for $\bx^{\star}$. 
Next, we claim that
   there exists $k'\in\N$ and a neighborhood $\mathcal{O}_1$ of $\bx^{\star}$ such that \eqref{eq::ge}
   holds for all $\bx\in\mathcal{O}_1\cap[-C, C]^n$ and $k\ge k'$,
   which yields a contradiction.

   We first consider the case that $\bx^{\star}\in\widehat{\X}$. As $F(\bx)\succ 0$ on $\widehat{\X}$, 
   there exists $\varepsilon>0$
   and a neighborhood $\mathcal{O}_1$ of $\bx^{\star}$ such that $F(\bx)\succeq\varepsilon I_p$ 
   for all $\bx\in\mathcal{O}_1$.  
   We now prove that there exsits $k'\in\N$ such that for all $k\ge k'$, it holds
    \[\sum_{j=1}^s\int_{\Y}\left\langle\left(I_{q_j}-G_j(\bx,
   \by)/M\right)^{2k}, G_j(\bx, \by)\right\rangle \ud\nu(\by)\le \frac{\varepsilon}{2},
    \]
    for all $\bx\in[-C, C]^n$. 
    Fix a pair of $\bu\in[-C, C]^n$ and $\bw\in\Y$. 
    Let $\{\lambda^{(j)}_i\}_{i\in[q_j]}$ be the set of eigenvalues of $G_j(\bu, \bw)$. 
    Take the decomposition $G_j(\bu, \bw)=Q_jD_jQ_j^{\intercal}$ where $D_j=\diag{\lambda_1^{(j)},\ldots,\lambda^{(j)}_{q_j}}$
    and $Q_j\in\RR^{q_j\times q_j}$ with $Q_j^{\intercal}Q_j=Q_jQ_j^{\intercal}=I_{q_j}$. Then, 
    \begin{equation}\label{eq::ineq}
    \begin{aligned}
    \left\langle\left(I_{q_j}-G_j(\bu, \bw)/M\right)^{2k}, G_j(\bu, \bw)\right\rangle
    =&\left\langle\left(I_{q_j}-Q_j(D_j/M)Q_j^{\intercal}\right)^{2k}, Q_jD_jQ_j^{\intercal}\right\rangle\\
    =&\left\langle Q_j\left(I_{q_j}-D_j/M\right)^{2k}Q_j^{\intercal}, Q_jD_jQ_j^{\intercal}\right\rangle
    =M\sum_{i=1}^{q_j} \frac{\lambda_i^{(j)}}{M}\left(1-\frac{\lambda_i^{(j)}}{M}\right)^{2k}
    \end{aligned}.
    \end{equation}
Let $q'\coloneqq\max_{ j\in[s]}q_j$. Note that there exists $k'\in\N$ such that for all $k\ge k'$,
    \[
    \max\left\{\xi(1-\xi)^{2k} : \xi\in[0, 1]\right\}\le \frac{\varepsilon}{2q'sM\nu(\Y)}.
    \]
    Hence, for any pair of $\bx\in[-C, C]^n$ and $\by\in\Y$, 
    from \eqref{eq::ineq}
    we obtain that \begin{equation}\label{eq::epsi}
    \left\langle\left(I_{q_j}-G_j(\bx, \by)/M\right)^{2k}, G_j(\bx, \by)\right\rangle
%       \le M\sum_{\lambda_i^{(j)}>0} \frac{\lambda_i^{(j)}}{M}\left(1-\frac{\lambda_i^{(j)}}{M}\right)^{2k} 
    \le 
   Mq_j \frac{\varepsilon}{2q'sM\nu(\Y)}\le \frac{\varepsilon}{2s\nu(\Y)}
    \end{equation}
    for all $k\ge k'$, where $k'$ does not depends on the choice of $\bx\in[-C, C]^n$ and $\by\in\Y$.
    Therefore, for all $\bx\in[-C, C]^n$ and $k\ge k'$, we have 
    \[
   \sum_{j=1}^s\int_{\Y}\left\langle\left(I_{q_j}-G_j(\bx,
   \by)/M\right)^{2k}, G_j(\bx, \by)\right\rangle \ud\nu(\by)
   \le \sum_{j=1}^s\int_{\Y}\frac{\varepsilon}{2s\nu(\Y)}\ud\nu(\by)\le \frac{\varepsilon}{2}.
    \]
    Then, for all $\bx\in\mathcal{O}_1\cap[-C, C]^n$ and $k\ge k'$, 
    \[
    F(\bx)-\sum_{j=1}^s\int_{\Y}\lram{I_p\otimes \left(I_{q_j}-G_j(\bx,
   \by)/M\right)^{2k}}{G_j(\bx, \by)}{p}\ud\nu(\by)\succeq \varepsilon I_p-\frac{\varepsilon}{2}I_p=
   \frac{\varepsilon}{2} I_p\succ 0.
    \]
    
    Next, we consider the case that $\bx^{\star}\in[-C,C]^n\setminus\widehat{\X}$. There exists $j_0\in[s]$ 
    and $\by^{(0)}\in\Y$ such that $\lambda_{\min}(G_{j_0}(\bx^{\star}, \by^{(0)})<0$.
    By continuity, for some $\bar{\lambda}<0$, 
    there exists a neighborhood $\mathcal{O}_1$ (resp., $\mathcal{O}_2$) 
    of $\bx^{\star}$ (resp., $\by^{(0)}$) such that $\lambda_{\min}(G_{j_0}(\bx, \by))\le \bar{\lambda}$
    for all $\bx\in\mathcal{O}_1$ and $\by\in\mathcal{O}_2$.
    For any pair of $\bx\in\mathcal{O}_1\cap [-C, C]^n$ and $\by\in\mathcal{O}_2\cap\Y$,  
    letting $\lambda_{\min}^{(j_0)}\coloneqq \lambda_{\min}(G_{j_0}(\bx, \by))$,
    from \eqref{eq::ineq} we obtain that
    \begin{equation}\label{eq::lambda}
    \begin{aligned}
    \left\langle\left(I_{q_{j_0}}-G_{j_0}(\bx, \by)/M\right)^{2k}, G_{j_0}(\bx, \by)\right\rangle &\le
    \sum_{\lambda_i^{(j_0)}>0} \lambda_i^{(j_0)}+
    \sum_{\lambda_i^{(j_0)}<0} \lambda_i^{(j_0)}\left(1-\frac{\lambda_i^{(j_0)}}{M}\right)^{2k}\\
    &\le Mq'+\lambda_{\min}^{(j_0)}\left(1-\frac{\lambda_{\min}^{(j_0)}}{M}\right)^{2k}
    \le Mq'+\bar{\lambda}\left(1-\frac{\bar{\lambda}}{M}\right)^{2k}.\\
    \end{aligned}
    \end{equation}
    Since $\supp(\nu)=\Y$, we have $\nu(\Y\cap\mathcal{O}_2)>0$. 
    Note that for each $j\in[s]$, by \eqref{eq::epsi}, the maximal eigenvalue of 
    \[
    \lram{I_p\otimes \left(I_{q_j}-G_j(\bx, \by)/M\right)^{2k}}{G_j(\bx, \by)}{p}
   =\left\langle\left(I_{q_j}-G_j(\bx, \by)/M\right)^{2k}, G_j(\bx, \by)\right\rangle I_p
    \]
    could be uniformly bounded from above on $[-C, C]^n\times \Y$ for all $k\in\N$. Therefore, 
    as $[-C, C]^n$ and $\Y$ are compact, there exists $R\in\RR$ such that
    \begin{equation}\label{eq::F}
    \begin{aligned}
   &F(\bx)-\sum_{j\neq j_0}\int_{\Y}\lram{I_p\otimes \left(I_{q_j}-G_j(\bx,
   \by)/M\right)^{2k}}{G_j(\bx, \by)}{p}\ud\nu(\by) \\
   &-\int_{\Y\setminus\mathcal{O}_2}\lram{I_p\otimes \left(I_{q_{j_0}}-G_{j_0}(\bx,
   \by)/M\right)^{2k}}{G_{j_0}(\bx, \by)}{p}\ud\nu(\by) 
   -(Mq'\nu(\Y\cap\mathcal{O}_2)+1)I_p\succeq  RI_p 
    \end{aligned}
    \end{equation}
    for all $\bx\in[-C, C]^n$. Since $1-\bar{\lambda}/M>1$, there exists $k'\in\N$ such that 
   \begin{equation}\label{eq::barlam}
       \bar{\lambda}\left(1-\frac{\bar{\lambda}}{M}\right)^{2k}<\frac{R}{\nu(\Y\cap\mathcal{O}_2)}
   \end{equation}
   holds for all $k\ge k'$. Then, for all $\bx\in\mathcal{O}_1\cap[-C, C]^n$ and $k\ge k'$, 
    \begin{align*}
   &F(\bx)-\sum_{j=1}^s\int_{\Y}\lram{I_p\otimes \left(I_{q_j}-G_j(\bx,
   \by)/M\right)^{2k}}{G_j(\bx, \by)}{p}\ud\nu(\by) \\
   \succeq & -\int_{\Y\cap\mathcal{O}_2}\lram{I_p\otimes \left(I_{q_{j_0}}-G_{j_0}(\bx,
   \by)/M\right)^{2k}}{G_{j_0}(\bx, \by)}{p}\ud\nu(\by) 
   +(Mq'\nu(\Y\cap\mathcal{O}_2)+1+R)I_p \tag{by \eqref{eq::F}}\\
   \succeq & \left(-\int_{\Y\cap\mathcal{O}_2}\left(Mq'+\bar{\lambda}\left(1-\frac{\bar{\lambda}}{M}\right)^{2k}\right)\ud\nu(\by) 
   +Mq'\nu(\Y\cap\mathcal{O}_2)+1+R\right)I_p \tag{by \eqref{eq::lambda}}\\
   \succeq& \left(-\left(Mq'+\frac{R}{\nu(\Y\cap\mathcal{O}_2)}\right)\nu(\Y\cap\mathcal{O}_2)
   +Mq'\nu(\Y\cap\mathcal{O}_2)+1+R\right)I_p=I_p, \tag{by \eqref{eq::barlam}}
    \end{align*}
    which validates the claim and completes the proof.
\end{proof}

Using Proposition \ref{prop::main}, we can recover Theorem \ref{th::SHpsatz} in the case that $\Y$ is compact.

\begin{lemma}\label{lem::S}
   For any $h\in\QM^1(\bG, \nu)$ and SOS matrix $S(\bx)\in\bS[\bx]^p$, 
   we have $S(\bx)h(\bx)\in\QM^p(\bG, \nu)$. 
\end{lemma}
\begin{proof}
    Write $h=\sigma_0+\sum_{j=1}^s\int_{\Y}\langle S_j, G_j\rangle\ud\nu(\by)$ where
all $\sigma_0\in\RR[\bx]$ and $S_j\in\bS[\bx, \by]^{q_j}$ are SOS. Then
\[
S h=S\sigma_0+\sum_{j=1}^s S\int_{\Y}\langle S_j, G_j\rangle\ud\nu(\by)
=S\sigma_0+\sum_{j=1}^s \int_{\Y}\lram{S\otimes S_j}{G_j}{p}\ud\nu(\by).
\]
As $S\sigma_0\in\bS[\bx]^{p}$ and $S\otimes S_j\in\bS[\bx, \by]^{pq_j}$ are all SOS
matrices, we have $Sh\in\QM^p(\bG, \nu)$.
\end{proof}

\begin{corollary}\label{cor::MPU}
Suppose that $\Y$ is compact and Assumption \ref{assump2} holds.
If $F(\bx)\in\bS[\bx]^p$ is PD on $\widehat{\X}$, then $F\in\QM^p(\bG, \nu)$.
\end{corollary}
\begin{proof}
   Let $C>0$ be given in Assumption \ref{assump2}. Then, by Proposition \ref{prop::main},
   there exists $M>0$ and $\bar{k}\in\N$ such that 
   \[
   F(\bx)-\sum_{j=1}^s\int_{\Y}\lram{I_p\otimes \left(I_{q_j}-G_j(\bx,
   \by)/M\right)^{2k}}{G_j(\bx, \by)}{p}\ud\nu(\by)\succ 0 
   \]
   holds on the set $\{\bx\in\RR^n \mid \Vert\bx\Vert^2\le C\}$ for all $k\ge \bar{k}$. 
   By Scherer-Hol's Positivestellensatz (Theorem \ref{th::psatz}), there exist SOS matrices 
   $S_0, S_1\in\bS[\bx]^p$ such that 
   \[
   F(\bx)-\sum_{j=1}^s\int_{\Y}\lram{I_p\otimes \left(I_{q_j}-G_j(\bx,
   \by)/M\right)^{2k}}{G_j(\bx, \by)}{p}\ud\nu(\by)=S_0(\bx)+S_1(\bx)(C-\Vert\bx\Vert^2).
   \]
   By Lemma \ref{lem::S}, $S_1(\bx)(C-\Vert\bx\Vert^2)\in\QM^p(\bG, \nu)$ which implies that
   $F\in\QM^p(\bG, \nu)$.
\end{proof}

\subsection{Proof of Theorem \ref{th::homo}}\label{sec::prfThmhom}
\begin{proof}[Proof of Theorem \ref{th::homo}]
Let 
\[
\wt{\X}\coloneqq\left\{\bx\in\RR^n \mid 1-\Vert\bx\Vert^2=0,\ 
G(\bx, \by)\succeq 0, \ \forall \by\in \Y\right\}.
\]
Then $F(\bu)\succ 0$ for all $\bu\in\wt{\X}$. 
By Corollary \ref{cor::pos}, it holds that  
\[
	F(\bx)=S'_0(\bx)+\int_{\Y}\lram{S'(\bx,\by)}{G(\bx, \by)}{p}
 \ud\nu(\by)+H(\bx)(1-\Vert\bx\Vert^2),
 \]
where $S'_0\in\bS[\bx]^{p}$, $S'\in\bS[\bx, \by]^{pq}$ are SOS  matrices and
$H(\bx)\in\bS[\bx]^p$.
%, $\deg(S_0)$, $\deg\lram{S_j}{G}{m}$, 
%$\deg(h)\le 2k'$ for some $k'\in\N$.
Replacing $\bx$ by $\bx/\Vert\bx\Vert$ in the above equality yields 
\[
F(\bx)\Vert\bx\Vert^{-\deg{F}}
 =S'_0\left(\frac{\bx}{\Vert\bx\Vert}\right)
 +\int_{\Y}\lram{S'\left(\frac{\bx}{\Vert\bx\Vert}, \by\right)}{G\left(\frac{\bx}{\Vert\bx\Vert}, \by\right)}{p}
 \ud\nu(\by).
\]
Let 
\[
k'\coloneqq\max\left\{\deg{F},\ \deg S'_0,\ \deg_{\bx}{G}+\deg_{\bx}S'\right\}.
\]
By assumption, $k'$ is even.
Multiplying the two sides of the last equality with $\Vert\bx\Vert^{k'}$ gives
\begin{equation}\label{eq::h}
F(\bx)\Vert\bx\Vert^{k'-\deg{F}}
 =\overline{S}_0\left(\bx\right)
 +\int_{\Y}\lram{\overline{S}\left(\bx, \by\right)}{G\left(\bx, \by\right)}{p}
 \ud\nu(\by),
\end{equation}
where
\[
\overline{S}_0\coloneqq S'_0\left(\frac{\bx}{\Vert\bx\Vert}\right)\Vert\bx\Vert^{k'}\ \
\text{and}\ \
\overline{S}\coloneqq S'\left(\frac{\bx}{\Vert\bx\Vert}, \by\right)\Vert\bx\Vert^{k'-\deg_{\bx}{G}}.
\]
%For simplicity, let $G_0\coloneqqI_m$ and $d_{G_0}\coloneqq0$. Let $j\in\{0, 1, \ldots, s\}$ be fixed.
Since $S'$ is an SOS matrix and $k'-\deg_{\bx}{G}\ge\deg_{\bx}S'$, 
we have $\overline{S}=H^{\intercal}H$ with
$H=H_{1}+H_{2}\Vert\bx\Vert$ where $H_{1}, H_{2}\in\RR[\bx,\by]^{\ell\times pq}$ for some $\ell\in\N$, 
are homogeneous in $\bx$ of degree $(k'-\deg_{\bx}{G})/2$ and $(k'-\deg_{\bx}{G})/2-1$, respectively. Thus,
\[
\begin{aligned}
\overline{S}=H^{\intercal}H=(H_{1}+H_{2}\Vert\bx\Vert)^\intercal
(H_{1}+H_{2}\Vert\bx\Vert)
=\left(H_{1}^\intercal H_{1}+H_{2}^\intercal H_{2}\Vert\bx\Vert^2\right)+
\left(H_{1}^\intercal H_{2}+H_{2}^\intercal H_{1}\right)\Vert\bx\Vert.
\end{aligned}
\]
Similarly, there exist homogeneous polynomial matrices $H_{0,1}\in\RR[x]^{\ell_0\times p}_{k'/2}$ 
and $H_{0,2}\in\RR[\bx]^{\ell_0\times p}_{k'/2-1}$ for some $\ell_0\in\N$ such that 
\[
\overline{S}_0=\left(H_{0,1}^\intercal H_{0,1}+H_{0,2}^\intercal H_{0,2}\Vert\bx\Vert^2\right)+
\left(H_{0,1}^\intercal H_{0,2}+H_{0,2}^\intercal H_{0,1}\right)\Vert\bx\Vert.
\]
Then, by \eqref{eq::h}, it holds that
\[
\begin{aligned}
F(\bx)\Vert\bx\Vert^{k'-\deg{F}}
=&\left(H_{0,1}^\intercal H_{0,1}+H_{0,2}^\intercal H_{0,2}\Vert\bx\Vert^2\right)
 +\int_{\Y}\lram{\left(H_{1}^\intercal H_{1}+H_{2}^\intercal H_{2}\Vert\bx\Vert^2\right)}{G\left(\bx, \by\right)}{p} \ud\nu(\by)\\
&+\left(\left(H_{0,1}^\intercal H_{0,2}+H_{0,2}^\intercal H_{0,1}\right)
 +\int_{\Y}\lram{\left(H_{1}^\intercal H_{2}+H_{2}^\intercal H_{1}\right)}{G\left(\bx, \by\right)}{p} \ud\nu(\by)\right)\Vert\bx\Vert.
\end{aligned}
\]
Since $\Vert\bx\Vert$ is not a polynomial and the left hand side of the above equation 
is a polynomial matrix,  we must have 
\[
\left(H_{0,1}^\intercal H_{0,2}+H_{0,2}^\intercal H_{0,1}\right)
 +\int_{\Y}\lram{\left(H_{1}^\intercal H_{2}+H_{2}^\intercal H_{1}\right)}{G\left(\bx, \by\right)}{p} \ud\nu(\by)=0.
\]
Then, letting $N\coloneqq(k'-\deg{F})/2\in\N$, 
$S_0\coloneqq H_{0,1}^\intercal H_{0,1}+H_{0,2}^\intercal H_{0,2}\Vert\bx\Vert^2$ 
and $S\coloneqq H_{1}^\intercal H_{1}+H_{2}^\intercal H_{2}\Vert\bx\Vert^2$,
we obtain the desired conclusion.
\end{proof}

\subsection{Proof of Theorem \ref{th::PolyaMPU}}\label{sec::prfThPolya}
Let $\bz=(z_1,\ldots,z_n)$ and $\bz\circ\bz=(z_1^2,\ldots,z_n^2)$.
We first derive an intermediary result.
\begin{theorem}\label{th::homo2}
Suppose that Assumption \ref{assump::carleman} holds, 
$F\in\bS[\bx]^p$ and $G\in\bS[\bx, \by]^{q}$ are homogeneous in $\bx$, and 
$F(\bx)\succ 0$ for all $\bx\in\RR^n_+\cap\X\setminus\{\mathbf{0}\}$. Then, there exists $N\in\N$ and 
polynomial matrices $S_0\in\bS[\bx]^p$, $S\in\bS[\bx,\by]^{pq}$ which 
are homogeneous in $\bx$ ($\deg S_0=N+\deg F$ and $\deg_{\bx}S=N+\deg F-\deg_{\bx}G$), and satisfy that $S_0(\bz\circ\bz)$ (resp., $S(\bz\circ\bz, \by)$)
are SOS matrices in $\bz$ (resp., $\bz$ and $\by$), such that 
\[
(\be^\intercal\bx)^{N}F(\bx)=S_0(\bx)+\int_{\Y}\lram{S(\bx, \by)}{G(\bx, \by)}{p}
 \ud\nu(\by).
 \]
\end{theorem}
\begin{proof}
    Consider the polynomial matrices $F(\bz\circ\bz)\in\bS[\bz]^p$ and 
    $G(\bz\circ\bz, \by)\in\bS[\bz, \by]^{q}$. Then, by assumption and Theorem \ref{th::homo},
there exists $N$ and polynomial matrices 
$H_0\in\RR[\bz]^{\ell_0\times p}$, $H\in\RR[\bz,\by]^{\ell\times pq}$ 
for some $\ell_0,\ell\in\N$ that are homogeneous in $\bz$, such that 
\[
\Vert\bz\Vert^{2N}F(\bz\circ\bz)
=(H_0^\intercal H_0)(\bz)+\int_{\Y}\lram{(H^\intercal H)(\bz,\by)}{G(\bz\circ\bz, \by)}{p}
 \ud\nu(\by).
 \]    
Moreover, by Theorem \ref{th::homo}, $\deg H_0=N+\deg F$, and $\deg_{\bx} H=N+\deg F-\deg_{\bx}G$.
Note that there are sets of polynomial matrices 
\[
\{H_{0,\ba}(\bx):\ba\in\{0,1\}^n\}\subset\RR[\bx]^{\ell_0\times p}\quad\text{and}\quad
\{H_{\ba}(\bx,\by): \ba\in\{0,1\}^n\}\subset\RR[\bx,\by]^{\ell\times p}
\]
which are homogeneous in $\bx$, such that 
\[
H_0(\bz)=\sum_{\ba\in\{0,1\}^n}\bz^{\ba}H_{0,\ba}(\bz\circ\bz),\quad
H(\bz, \by)=\sum_{\ba\in\{0,1\}^n}\bz^{\ba}H_{\ba}(\bz\circ\bz, \by).
\]
Then
\[
\begin{aligned}
&(\be^\intercal(\bz\circ\bz))^{N}F(\bz\circ\bz)\\
=&\sum_{\ba\in\{0,1\}^n}(\bz\circ\bz)^{\ba}(H_{0,\ba}^\intercal H_{0,\ba})(\bz\circ\bz)
+\int_{\Y}\lram{\sum_{\ba\in\{0,1\}^n}(\bz\circ\bz)^{\ba}(H_{\ba}^\intercal H_{\ba})(\bz\circ\bz, \by)}{G(\bz\circ\bz, \by)}{p}
 \ud\nu(\by)\\
+&\sum_{\ba,\bb\in\{0,1\}^n, \ba\neq\bb}\bz^{\ba+\bb}(H_{0,\ba}^\intercal H_{0,\bb})(\bz\circ\bz)
+\int_{\Y}\lram{\sum_{\ba,\bb\in\{0,1\}^n, \ba\neq\bb}\bz^{\ba+\bb}(H_{\ba}^\intercal H_{\bb})(\bz\circ\bz, \by)}{G(\bz\circ\bz, \by)}{p}
 \ud\nu(\by).
\end{aligned}
 \]  
Comparing the even and odd terms in $\bz$ in the above equation, we get
\[
\begin{aligned}
&(\be^\intercal(\bz\circ\bz))^{N}F(\bz\circ\bz)\\
=&\sum_{\ba\in\{0,1\}^n}(\bz\circ\bz)^{\ba}(H_{0,\ba}^\intercal H_{0,\ba})(\bz\circ\bz)
+\int_{\Y}\lram{\sum_{\ba\in\{0,1\}^n}(\bz\circ\bz)^{\ba}(H_{\ba}^\intercal H_{\ba})(\bz\circ\bz, \by)}{G(\bz\circ\bz, \by)}{p}
 \ud\nu(\by).
\end{aligned}
\]
It follows that the equality 
\[
\begin{aligned}
(\be^\intercal\bx)^{N}F(\bx)
=\sum_{\ba\in\{0,1\}^n}\bx^{\ba}H_{0,\ba}^\intercal(\bx) H_{0,\ba}(\bx)
+\int_{\Y}\lram{\sum_{\ba\in\{0,1\}^n}\bx^{\ba}H_{\ba}^\intercal(\bx,\by) H_{\ba}(\bx, \by)}{G(\bx, \by)}{p} \ud\nu(\by)
\end{aligned}
\]
holds for all $\bx\in\RR^n_+$. As $\RR^n_+$ has interior points, the polynomial matrices 
on the left and right side of the above equality are identical.
Letting 
\[
S_0(\bx)\coloneqq\sum_{\ba\in\{0,1\}^n}\bx^{\ba}H_{0,\ba}^\intercal(\bx) H_{0,\ba}(\bx)
\text{ and }
S(\bx,\by)\coloneqq\sum_{\ba\in\{0,1\}^n}\bx^{\ba}H_{\ba}^\intercal(\bx,\by) H_{\ba}(\bx,\by),
\]
the conclusion follows.
\end{proof}

%we can 
%Write $P(\bx)=\sum_{\ba\in\supp{P}}P_{\ba}\bx^{\ba}$  and 
%$Q(\bx,\by)=\sum_{\ba\in\suppx{\bx}{Q}}Q_{\ba}(\by)\bx^{\ba}$ for some 
%$P_{\ba}\in\bS^p$ and $Q_{\ba}(\by)\in\bS[\by]^q$. 
As an extension of P{\'o}lya's result \cite{polya},
Scherer and Hol \cite{SH2006} provided the following certificate for homogeneous polynomial 
matrices being positive on the nonnegative orthant. 
\begin{theorem}\label{th::Polya}\cite[Theorem 3]{SH2006}
   Suppose that the polynomial matrix $H(\bx)\in\bS[\bx]^p$ is homogeneous and $H(\bx)\succeq\lambda I_p$ 
   for some $\lambda>0$ on $\{\bx\in\RR_+^n\mid \be^\intercal\bx=1\}$. Write $H(\bx)=\sum_{\ba\in\supp(H)}H_{\ba}\bx^{\ba}$
   with each $H_{\ba}\in\bS^p$ and let 
   \[
   L(H)\coloneqq\max_{\ba\in\supp(H)}\frac{\ba!}{\deg H!}\Vert H_{\ba}\Vert, 
   \]
   where $\Vert\cdot\Vert$ denotes the spectral norm. Then, for all 
   \[
   N\ge\frac{\deg H(\deg(H)-1)L(H)}{2\lambda}-\deg H, 
   \]
   the coefficient of $\bx^{\ba}$ in $(\be^\intercal\bx)^NH(\bx)$ for each $\ba\in\N^n_{N+\deg H}$ is PD.
\end{theorem}
We will use this theorem as another intermediary result. Let  $D\coloneqq\max\left\{\deg F, \deg_{\bx}G\right\}$. 
\begin{lemma}\label{lem::poshom}
Suppose that 
$F\in\bS[\bx]^p, G\in\bS[\bx, \by]^{q}$ are homogeneous in $\bx$ and 
$F(\bx)\succ 0$ for all $\bx\in\RR^n_+\cap\X\setminus\{\mathbf{0}\}$. 
Then, there exists $\varepsilon>0$ such that the homogeneous polynomial matrix
\[
F_{\varepsilon}(\bx)\coloneqq(\be^\intercal\bx)^{D-\deg F}F(\bx)-\varepsilon\left((\be^\intercal\bx)^DI_p+
\int_{\Y}\lram{(\be^\intercal\bx)^{D-\deg_{\bx}G}I_{pq}}{G(\bx, \by)}{p}\ud\nu(\by)\right)\succ 0 
\]
for all $\bx\in\RR^n_+\cap\X\setminus\{\mathbf{0}\}$.
\end{lemma}
\begin{proof}
%Since $F_{\varepsilon}$ is 
By homogeneity, we only need to prove $F_{\varepsilon}\succ 0$ on 
    $\overline{\X}\coloneqq\RR^n_+\cap\X\cap\{\bx\in\RR^n \mid \be^\intercal\bx=1\}.$
%\[
%\overline{\X}\coloneqq\RR^n_+\cap\X\cap\{\bx\in\RR^n \mid \be^\intercal\bx=1\}.
%\]
That is, there exists $\varepsilon>0$ such that 
$F(\bx)-\varepsilon\left(I_p+\int_{\Y}\lram{I_{pq}}{G(\bx, \by)}{p}\ud\nu(\by)\right)\succ 0$
%\[
%F(\bx)-\varepsilon\left(I_p+\int_{\Y}\lram{I_{pq}}{G(\bx, \by)}{p}\ud\nu(\by)\right)\succ 0 
%\]
for all $\bx\in\overline{\X}$. Since $\overline{\X}$ is compact and $F(\bx)\succ 0$ 
on $\RR^n_+\cap\X\setminus\{\mathbf{0}\}$, letting 
\begin{equation}\label{eq::epsilon}
\varepsilon\coloneqq\frac{\min_{\bx\in\overline{\X}}\lambda_{\min}(F(\bx))}
{2\max_{\bx\in\overline{\X}}\vert\lambda_{\max}\left(I_p+\int_{\Y}\lram{I_{pq}}{G(\bx, \by)}{p}\ud\nu(\by)\right)\vert},
\end{equation}
we obtain the desired conclusion.
\end{proof}

We now prove Theorem \ref{th::PolyaMPU}. 
\begin{proof}[Proof of Theorem \ref{th::PolyaMPU}]
Let $F_{\varepsilon}(\bx)$ be the polynomial matrix in Lemma \ref{lem::poshom} where $\varepsilon>0$
is defined in \eqref{eq::epsilon}. Then,
applying Theorem \ref{th::homo2} to $F_{\varepsilon}(\bx)$ yields $N_1\in\N$ and 
polynomial matrices $S'_0\in\bS[\bx]^p$, $S'\in\bS[\bx,\by]^{pq}$ which 
are homogeneous in $\bx$ and satisfy that $S'_0(\bz\circ\bz)$, $S'(\bz\circ\bz, \by)$'s 
are SOS matrices, such that 
\[
(\be^\intercal\bx)^{N_1}F_{\varepsilon}(\bx)=S'_0(\bx)+\int_{\Y}\lram{S'(\bx, \by)}{G(\bx, \by)}{p}\ud\nu(\by).
 \]
Then, by the definition of $F_{\varepsilon}(\bx)$,
\begin{equation}\label{eq::gamma}
(\be^\intercal\bx)^{D-\deg F+N_1}F(\bx)
=\Gamma_0(\bx)+\int_{\Y}\lram{\Gamma(\bx, \by)}{G(\bx, \by)}{p}\ud\nu(\by) ,
\end{equation}
where 
\[
\Gamma_0(\bx)\coloneqq\varepsilon(\be^\intercal\bx)^{D+N_1}I_p+S'_0(\bx),\ \
\Gamma(\bx, \by)\coloneqq\varepsilon(\be^\intercal\bx)^{D-\deg_{\bx}G+N_1}I_{pq}+S'(\bx, \by).
\]
Note that $\Gamma_0(\bx)$ and $\Gamma(\bx, \by)$ are homogeneous in $\bx$ because 
by Theorem \ref{th::homo2}, 
\[
\deg S'_0=N_1+\deg F_{\varepsilon}=D+N_1 \ \ \text{and} \ \ 
\deg_{\bx} S'=N_1+\deg F_{\varepsilon}-\deg_{\bx}G=D-\deg_{\bx}G+N_1.
\]
Since $S'_0(\bz\circ\bz)$, $S'(\bz\circ\bz, \by)$ are SOS matrices,
for all $\bx\in\{\bx\in\RR_+^n\mid \be^\intercal\bx=1\}$, it holds
\[
\Gamma_0(\bx)\succeq\varepsilon I_p\ \text{ and }\
\Gamma(\bx, \by)\succeq\varepsilon I_{pq}, \  \forall \by\in\Y.
\]
Write $\Gamma(\bx, \by)=\sum_{\ba\in\suppx{\bx}{\Gamma}}\Gamma_{\ba}(\by)\bx^{\ba}$
and define
\[
   L(\Gamma)\coloneqq\max_{\by\in\Y}\max_{\ba\in\suppx{\bx}{\Gamma}}\frac{\ba!}{\deg_{\bx}\Gamma!}\Vert \Gamma_{\ba}(\by)\Vert.
\]
As $\Y$ is compact, $L(\Gamma)$ is well-defined. Let
\[
\begin{aligned}
   N_2\coloneqq\max\left\{\frac{\deg\Gamma_0(\deg\Gamma_0-1)L(\Gamma_0)}{2\varepsilon}-\deg\Gamma_0, 
   \ \frac{\deg_{\bx}\Gamma(\deg_{\bx}\Gamma-1)L(\Gamma)}{2\varepsilon}-\deg_{\bx}\Gamma\right\}.
\end{aligned}
   \]
Let $P_0(\bx)=(\be^\intercal\bx)^{N_2}\Gamma_0$ and $P(\bx, \by)=(\be^\intercal\bx)^{N_2}\Gamma$.
Write 
\[
P_0(\bx)=\sum_{|\ba|=\deg P_0}P_{0,\ba}\bx^{\ba}  \quad\text{ and } \quad
P(\bx, \by)=\sum_{|\ba|=\deg_{\bx}P}P_{\ba}(\by)\bx^{\ba}.
\]
By Theorem \ref{th::Polya},  $P_{0,\ba}\succ 0$ for all $|\ba|=\deg P_0$, 
$P_{\ba}(\by)\succ 0$ for all $|\ba|=\deg_{\bx}P$ and $\by\in\Y$.
Letting $N=D-\deg F+N_1+N_2$, the conclusion follows from \eqref{eq::gamma}.
\end{proof}

\section{Conclusions}\label{sec::cons}
We extend various classical Positivstellens\"atze to provide SOS-structured characterizations 
for polynomial matrices that are positive (semi)definite over a
semialgebraic set defined by a PMI with UQs. 
Under the Archimedean condition, we first present a matrix-valued Positivstellensatz 
incorporating universal quantifiers, 
along with a sparse version for scalar-valued polynomial objectives, leveraging the correlative sparsity patterns. 
We also derive two generalized Positivstellens\"atze without assuming the Archimedean condition. 
These results significantly extend the existing work on Positivstellens\"atze, and some of them remain novel and intriguing even in the absence of UQs. 
Our results are valuable in ensuring the robustness of PMIs over a prescribed set with uncertainty, allowing potential applications in areas such as optimal control, systems theory, and certifying stability.

%\section{}
%
%Let $f(\bx)\in\RR[\bx]$, $g_j(\bx, \by^{(j)})\in\RR[\bx, \by^{(j)}]$, $j=1,\ldots,s.$
%
%\begin{equation}\label{defineX}
%	\X\coloneqq\{\bx\in\RR^n \mid g_j(\bx, \by^{(j)})\ge 0, \ \forall  
% \by^{(j)}\in \Y_j\subset\RR^{m_j},\ j\in[s]\}.
%\end{equation}
%Let $\nu=(\nu_1,\ldots,\nu_s)$ where each $\nu_j$ is a Borel measure on $\RR^{m_j}$ 
%such that $\supp{\nu_j}=\Y_j$. Let $\bg\coloneqq\{g_1,\ldots,g_s\}$.
%For each $k\in\N$, we define the $k$-th \emph{truncated quadratic module} 
%$\QM_k(\bg, \nu)$ generated by $\bg$ and $\nu$ by 
%\[
%	\QM_k(\bg,
%	\nu)\coloneqq\left\{S_0+\sum_{j=1}^s\int_{\Y_j}S_j g_j\ud\nu_j(\by^{(j)})\ \middle\vert \
%\begin{aligned}
%&S_0\in\RR[\bx], S_j\in\RR[\bx, \by^{(j)}],\\
%&\sigma_0, \sigma_j\ \text{are SOS polynomials},\\
%&\deg(\sigma_0), \deg(\sigma_j g_j)\le 2k
%\end{aligned}
%\right\},
%\]
%and define the \emph{quadratic module} by 
%\[
%\QM(\bg, \nu)\coloneqq\bigcup_{k\in\N}\QM_k(\bg, \nu).
%\]
%\begin{assumption}\label{assump0}
%{\rm
%There is $M>0$ such that $M - \Vert \bx\Vert^2\in\QM(\bg, \nu)$.}
%\end{assumption}

\section*{Acknowledgements}
The authors appreciate Professor Konrad Schm{\" u}dgen for providing the result in \cite[Theorem 10.25]{Sch2020}, which significantly simplified the proof of Theorem \ref{th::SHpsatz}.
Feng Guo was supported by the Chinese National Natural Science Foundation under grant 12471478. 
Jie Wang was supported by National Key R\&D Program of China under grant No. 2023YFA1009401, the Strategic Priority Research Program of the Chinese Academy of Sciences XDB0640000 \& XDB0640200, and the National Natural Science Foundation of China under grant No. 12201618 \& 12171324. 

\bibliographystyle{abbrv}
\bibliography{mpu}

\end{document}